\documentclass[article,11pt]{amsart}
\usepackage[top=25mm,bottom=25mm,left=25mm,right=25mm]{geometry}
\usepackage{amsmath,amssymb}
\usepackage{graphicx,times,bbm,url}
\usepackage{booktabs}
\usepackage{rotating}

\newtheorem{lemma}{Lemma}[section]
\newtheorem{proposition}{Proposition}[section]
\newtheorem{theorem}{Theorem}[section]
\newtheorem{corollary}{Corollary}[section]

\newtheorem{example}{Example}[section]
\newtheorem{remark}{Remark}[section]

\makeatletter
\def\section{\@startsection{section}{1}%
\z@{1\linespacing\@plus\linespacing}{1\linespacing}%
{\bf\centering}}
\def\subsection{\@startsection{subsection}{0}%
\z@{\linespacing\@plus\linespacing}{\linespacing}%
{\bf}}
\makeatother

\makeatletter
\@addtoreset{equation}{section}
\makeatother

\DeclareMathOperator{\supp}{supp}

\DeclareMathOperator{\loc}{loc}

\DeclareMathOperator{\Id}{Id}

\providecommand{\pro}[1]{(#1_t)_{t \geq 0}}

\newcommand{\cK}{\mathcal{K}}

\newcommand{\cE}{\mathcal{E}}

\newcommand{\R}{\mathbb{R}}
\newcommand{\1}{\mathbf{1}}
\newcommand{\Z}{\mathbb{Z}}
\newcommand{\pr}{\mathbf{P}}
\newcommand{\qpr}{\mathbb{Q}}

\newcommand{\ex}{\mathbf{E}}

\newcommand{\N}{\mathbb{N}}

\usepackage{color}

\begin{document}
\title[Quenched behaviour]
{The quenched asymptotics for nonlocal Schr\"{o}dinger operators with Poissonian potentials}
\author{Kamil Kaleta and Katarzyna Pietruska-Pa{\l}uba}

\address{K. Kaleta \\ Institute of Mathematics, University of Warsaw, ul. Banacha 2, 02-097 Warszawa and Faculty of Pure and Applied Mathematics, Wroc{\l}aw University of Technology, Wyb. Wyspia\'nskiego 27, 50-370 Wroc{\l}aw, Poland}
\email{kamil.kaleta@pwr.edu.pl, kkaleta@mimuw.edu.pl}

\address{K. Pietruska-Pa{\l}uba \\ Institute of Mathematics \\ University of Warsaw
\\ ul. Banacha 2, 02-097 Warszawa, Poland}
\email{kpp@mimuw.edu.pl}
\begin{abstract}
{We study the quenched long time behaviour of the survival probability up to time $t$, \linebreak $\ex_x\big[e^{-\int_0^t V^{\omega}(X_s){\rm d}s}\big],$  of a symmetric L\'evy process with jumps, under a sufficiently regular Poissonian random potential $V^{\omega}$ on $\R^d$. Such a function is a probabilistic solution to the parabolic equation involving the nonlocal Schr\"odinger operator based on the generator of the process $(X_t)_{t \geq 0}$ with potential $V^{\omega}$. For a large class of processes and potentials of finite range, we determine  rate functions $\eta(t)$ and compute explicitly the positive constants $C_1, C_2$ such that
\[-C_1 \leq \liminf_{t \to \infty} \frac{\log \ex_x\big[{\rm e}^{-\int_0^t V^{\omega}(X_s){\rm d}s}\big]}{\eta(t)}
 \leq \limsup_{t \to \infty}
 \frac{\log \ex_x\big[{\rm e}^{-\int_0^t V^{\omega}(X_s){\rm d}s}\big]}{\eta(t)} \leq -C_2, \]
almost surely with respect to $\omega$, for every fixed $x \in \R^d$. The functions $\eta(t)$ and the bounds $C_1, C_2$ heavily depend on the intensity of large jumps of the process. In particular, if its decay at infinity is `sufficiently fast', then we prove that $C_1=C_2$, i.e. the limit exists. Representative examples in this class are relativistic stable processes with L\'evy-Khintchine exponents $\psi(\xi) = (|\xi|^2+m^{2/\alpha})^{\alpha/2}-m$, $\alpha \in (0,2)$, $m>0$, for which we obtain that
\[\lim_{t \to \infty} \frac{\log \ex_x\big[{\rm e}^{-\int_0^t V^{\omega}(X_s)ds}\big]}{t/(\log t)^{2/d}} =  \frac{\alpha}{2}m^{1-\frac{2}{\alpha}} \, \left(\frac{\rho \omega_d}{d}\right)^{\frac{d}{2}} \, \lambda_1^{BM}(B(0,1)), \quad \text{for almost all $\omega$,}\]
where $\lambda_1^{BM}(B(0,1))$ is the principal eigenvalue of the Brownian motion killed on leaving the unit ball, $\omega_d$ is the Lebesgue measure of a unit ball and $\rho>0$ corresponds to $V^{\omega}$. We also identify  two interesting regime changes ('transitions') in the growth properties of the rates $\eta(t)$ as the intensity of large jumps of the processes varies from polynomial to higher order, and eventually to stretched exponential order.
}


\bigskip
\noindent
2010 {\it MS Classification}: Primary 60G51, 60H25, 60K37; Secondary 47D08, 60G60, 47G30

\end{abstract}

\footnotetext{
Research supported by the National Science Centre (Poland) internship
grant on the basis of the decision No. DEC-2012/04/S/ST1/00093 and by the Foundation for Polish Science.}

\maketitle

\baselineskip 0.5 cm


\section{Introduction}
This paper is concerned with the large time asymptotic behaviour of the solutions
of the spatially continuous parabolic nonlocal Anderson problem with Poissonian interaction,
driven by a L\'evy process in $\R^d$.
More precisely, we consider the equation
\begin{equation}\label{eq:intro1}
\partial_t u = Lu-V^\omega u, \quad  u(0,x)\equiv 1,
\end{equation}
where $L$ is the generator of the underlying process
and $V^\omega(x)=\int_{\mathbb R^d} W(x-y)\mu^\omega({\rm d}y)$ is a random Poissonian potential with
sufficiently regular profile function $W:\mathbb R^d\to\mathbb R_+$. By $\mu^\omega$ we denote the
Poisson random measure on $\mathbb R^d$ with intensity  $\rho \,{\rm d}x$, $\rho>0$,
over a given probability space $(\Omega, \qpr)$.

Processes considered  throughout the paper, $X=(X_t, \pr_x)_{t \geq 0,\ x \in \R^d},$ are symmetric L\'evy processes {with jumps}, with
characteristic functions
$$
\ex_0 \left[e^{i \xi \cdot X_t}\right] = e^{-t \psi(\xi)}, \quad \xi \in \R^d, \ t>0,
$$
whose characteristic exponents {(symbols)} $\psi$ are given by the L\'evy-Khintchine formula
\begin{align} \label{eq:Lchexp}
\psi(\xi) = \xi \cdot A \xi + \int_{\R^d\setminus \{0\}} (1-\cos(\xi \cdot z)) \nu({\rm d}z).
\end{align}
Here $A=(a_{ij})_{1 \leq i,j \leq d}$ is a symmetric non-negative definite matrix, and $\nu$ is a symmetric L\'evy
measure, {i.e. a Radon measure on $\R^d \backslash \left\{0\right\}$} that satisfies $\int_{\R^d\setminus \{0\}} (1 \wedge |z|^2) \nu({\rm d} z) < \infty$ and
$\nu(E)= \nu(-E)$, for every Borel $E \subset \R^d \backslash \left\{0\right\}$ \cite{bib:J, bib:JSch}.
{We always assume that $X$ is strong Feller and ${\rm e}^{-t_0 \psi(\cdot)} \in L^1(\R^d)$, for some $t_0>0$
(for more details see Section \ref{subsec:levy-proc}).}

Since its introduction in the 50's of the past century, the  parabolic Anderson model  based on the Laplacian (both continuous and discrete), with various potentials,  has been studied with varying intensity. For an excellent review of the history of the research in
this area we refer to the book of K\"{o}nig \cite{bib:Kon-Wolff}.

Under suitable regularity assumptions, the solution to the problem \eqref{eq:intro1} can be probabilistically represented by means
of the Feynman-Kac formula:
\begin{equation}\label{eq:functional-u}
u^\omega(t,x)=\mathbf E_x\left[{\rm e}^{-\int_0^t V^\omega(X_s)\,{\rm d}s}\right].
\end{equation}

One is interested in the long-time behaviour of $u^\omega(t,x),$
in both the  annealed sense
(averaged with respect to $\mathbb Q$) and the  quenched
sense (almost sure with respect to $\mathbb Q$).
In this paper, we will analyse the quenched behaviour of functionals $u^\omega(t,x)$ for L\'{e}vy processes whose exponent $\psi$
can be written as
\begin{equation}\label{eq:psi-okura}
\psi(x)=\psi^{(\alpha)}(x) +o(|x|^\alpha), \quad |x|\to 0,
\end{equation}
for some $\alpha \in (0,2],$
 and satisfies some mild assumptions concerning its behaviour at infinity. In formula (\ref{eq:psi-okura}), $\psi^{(\alpha)}$ is the
 characteristic exponent of a symmetric (not necessarily isotropic) $\alpha-$stable process,  i.e.  a L\'evy process with characteristic exponent
\begin{equation}\label{eq:psi-stable}\psi^{(\alpha)}(\xi)=\int_0^\infty\int_{S^{d-1}} \frac{1-\cos(\xi\cdot rz)}{r^{1+\alpha}}\,n({\rm d}z){\rm d}r,
\end{equation}
where $n$ is a symmetric finite measure on unit sphere $S^{d-1}$ when $\alpha\in(0,2),$ or
\begin{equation}\label{eq:psi-diff}
\psi^{(\alpha)}(\xi)= \xi \cdot A \xi, 
\end{equation} where $A=(a_{ij})_{1 \leq i,j \leq d}$ is a symmetric nonnegative definite matrix when $\alpha=2.$
When $n$ is the uniform distribution on $S^{d-1}$ for $\alpha \in (0,2)$ or $A \equiv a \Id$ with some $a>0$ for $\alpha = 2$, then the process is called \emph{isotropic $\alpha-$stable.} We assume the nondegeneracy condition $\inf_{|\xi|=1} \psi^{(\alpha)}(\xi) > 0$.

The annealed
asymptotics of $u^\omega(t,x)$ has been {first analyzed} by Donsker-Varadhan \cite{bib:Don-Var} (for stable processes, including the Brownian motion) and of Okura {\cite{bib:Okura81}} (for {symmetric} L\'{e}vy processes satisfying (\ref{eq:psi-okura})). When the profile $W$ is of order $o(1/|x|^{d+\alpha})$ when $|x|\to\infty,$
they prove that
\begin{equation}\label{eq:dv}
\lim_{t\to\infty}\frac{\mathbb {\log E_Q}\left[u^\omega(t,x)\right]}{t^{d/(d+{{\alpha}})}}= - (\rho \omega_d)^{\frac{\alpha}{d+\alpha}}\left(\frac{d+\alpha}{\alpha}\right)\left(\frac {2{\lambda_{(\alpha)}}}{d}\right)^{\frac{d}{d+\alpha}},
\end{equation}
In this formula, $\omega_d$ is the volume of the unit ball, and \[
{\lambda_{(\alpha)}}=\inf_{U\,\mbox{\small open},\, |U|=\omega_d }{\lambda^{(\alpha)}_1(U)}\] denotes
the infimum of principal eigenvalues for the symmetric $\alpha-$stable process with exponent (\ref{eq:psi-stable}) in $U$ with outer Dirichlet conditions on $U^c.$  Okura's work covers also the case when $\psi(x)=O(\psi^{(\alpha)}(x)),$ $|x|\to 0,$ but only when the potential is heavy-tailed.  This falls not within the scope of present paper and so we will {discuss this case elsewhere}.

The key observation used in the quenched case is that
when the profile function $W$ is of bounded support, then $\mathbb Q-$a.s. there exist large areas with no potential interaction.
Typically, with high probability,  the process
tends to remain in those `atypical', `favorable' areas, which affects the a.s. behaviour of the functional.  As a result,  the quenched behaviour can differ from the annealed asymptotics.

This phenomenon (for the Brownian motion only) was first observed and rigorously established by  Sznitman in \cite{bib:Szn-ptrf93}. He proves that in that case, for any $x\in\mathbb R^d$, and $\mathbb Q-$almost all $\omega,$
\begin{equation}\label{eq:szn-quenched}
\lim_{t\to\infty}\frac{\log u^\omega(t,x)}{t/(\log t)^{2/d}} = -\left(\frac{\rho \omega_d}{d}\right)^{\frac{2}{d}}\,{\lambda_1^{BM}(B(0,1))},
\end{equation} where ${\lambda_1^{BM}(B(0,1))}$ is  the principal eigenvalue
for the Brownian motion killed on exiting $B(0,1).$ This result was reproven by Fukushima \cite{bib:Fuk}.
For the Brownian motion on some irregular spaces such as the  Sierpi\'{n}ski gasket,
one also sees a similar phenomenon: rates of the annealed and the quenched asymptotics differ (see \cite{bib:kpp-ptrf,bib:kpp-spa}).

In this paper, we address the quenched asymptotics for L\'{e}vy processes with jumps influenced by potentials with compact-range profiles. Key examples include a vast selection of isotropic unimodal L\'evy processes, subordinate Brownian motions,
processes with nondegenerate Brownian components and with non-isotropic L\'evy densities
as well as processes with less regular L\'evy measures that have product or discrete large jumps components.
While the 'favorable' spots in the {Poissonian} configuration are still present,
the jumping nature of L\'{e}vy processes drives the process out of
those spots: if the process does not stay there long enough, then
the effect of `no-potential-interaction' is spoiled and as a consequence the quenched rate can be the same as the annealed rate.
What is decisive here is the intensity of long jumps of the process:
for processes with L\'{e}vy measures whose tails decay fast enough at infinity,
we see the same phenomenon as that for Brownian motion.

For more clarity, we have collected the results obtained for particular classes of processes with
various types of large jump intensities in Table 1 below
(for simplicity we restricted the presentation to the family of isotropic unimodal L\'evy processes with stable-type small jumps).

The annealed rate is always governed by the exponent $\alpha$ appearing in \eqref{eq:psi-okura}, which is determined by the
behaviour of the exponent of $\psi$ near zero. Formula \eqref{eq:psi-okura} together with some mild assumptions concerning the behaviour of the symbol at infinity permit to obtain the annealed asymptotics  of $u^\omega(t,x)$ and also
to identify the constant in (\ref{eq:dv}).

The question of the quenched rate is much more delicate.
{In this case, the formula \eqref{eq:psi-okura} (even if combined with some information on the behaviour of the characteristic exponent at infinity) is generally insufficient. This is particularly evident when $\alpha=2$. It occured to us as a surprise that the effective derivation of the quenched rate (and the corresponding bounds) requires deep analysis of the subtle properties of L\'{e}vy processes
with prescribed L\'{e}vy measures, depending on the type of their fall-off at infinity.}

As usual, in this paper the upper and the lower bounds of $u^\omega(t,x)$ are addressed separately. First,
 in Sections \ref{sec:upper} and \ref{sec:lower} we prove two general results: Theorem \ref{thm:upper} concerning the upper bound, and Theorem \ref{thm:lower_bound} concerning the lower bound.

 The rest
 of the paper (Section \ref{sec:specific}) is devoted to the
 application of our  general results  for
specific classes of processes.

(1) For processes {satisfying \eqref{eq:psi-okura}} with $\alpha\in(0,2)$ {(Theorem \ref{thm:polynomial_irr} and Examples \ref{ex:polynomial_irr}-\ref{ex:polynomial_irr_2})}, and also for those with $\alpha=2$ but {polynomially} decaying L\'{e}vy measures {(Theorem \ref{thm:polynomial} and Examples \ref{ex:polynomial}, \ref{ex:ex_irr} (2))}, the quenched and annealed rates coincide and are both equal
to $t^{d/(d+\alpha)}.$

\begin{sidewaystable}
\vspace{17cm}
\hspace{-0.4cm}
  \begin{tabular}{llrrrr}
    \toprule
    	\textbf{intensity of jumps / process} & \textbf{parameters} & \ \textbf{rate $\eta(t)$}  & \textbf{lower bound for $\liminf_{t \to \infty} \frac{\log u^{\omega}(t,x)}{\eta(t)}$} & \textbf{upper bound for $\limsup_{t \to \infty} \frac{\log u^{\omega}(t,x)}{\eta(t)}$} & \textbf{limit}\\
			    \midrule
		& &&& \\
    $\frac{C}{r^{d+\alpha}}$ & $\alpha \in (0,2)$
		& $t^{\frac{d}{d+\alpha}}$
		& $ - \frac{4\alpha+9d}{2}\left(\frac{2}{d+2\alpha}\right)^{\frac{d}{d+\alpha}}
        \left(\frac{\rho\omega_d}{d}\right)^{\frac{\alpha}{d+\alpha}}\left(\lambda_{1,\nu}^{(\alpha)}\right)^{\frac{d}{d+\alpha}}$
    & $ - \alpha \left(\frac{2}{d+2\alpha}\right)^{\frac{d}{d+\alpha}}
        \left(\frac{\rho\omega_d}{d}\right)^{\frac{\alpha}{d+\alpha}}\left(\lambda_{1,\nu}^{(\alpha)}\right)^{\frac{d}{d+\alpha}}$
		& no \\
		&&&& \\
    \midrule
		& $\alpha \in (0,2)$ &&& \\
    $\frac{C}{r^{d+\alpha}} \1_{\left\{r \leq 1\right\}} + \frac{C}{r^{d+\delta}} \1_{\left\{r > 1\right\}}$ & $\delta >2$
		& $t^{\frac{d}{d+2}}$
		& $ - \frac{4\delta+9d}{2}\left(\frac{2}{d+2\delta}\right)^{\frac{d}{d+2}}
        \left(\frac{\rho\omega_d}{d}\right)^{\frac{2}{d+2}}\left(\lambda_{1,\nu}^{(2)}\right)^{\frac{d}{d+2}}$
    & $ - \delta \left(\frac{2}{d+2\delta}\right)^{\frac{d}{d+2}}
        \left(\frac{\rho\omega_d}{d}\right)^{\frac{2}{d+2}}\left(\lambda_{1,\nu}^{(2)}\right)^{\frac{d}{d+2}}$
		& no \\
		&&&& \\
    \midrule
   	& $\alpha \in (0,2)$ &&& \\
   $\frac{C}{r^{d+\alpha}} \1_{\left\{r \leq 1\right\}} + \frac{C}{e^{\theta(\log r)^{\beta}}} \1_{\left\{r > 1\right\}}$ & $\theta > 0$
	  & $t^{\frac{d\beta}{d\beta+2}}$
		& $- 2 \theta^{\frac{2}{2 + d\beta}}\left(\frac{\omega_d \rho}{d}\right)^{\frac{2 \beta }{2 + d \beta}} \left(\lambda_{1,\nu}^{(2)}\right)^{\frac{d \beta }{2 + d \beta}}$
		& $- \theta^{\frac{2}{2 + d\beta}}\left(\frac{\omega_d \rho}{d}\right)^{\frac{2 \beta }{2 + d \beta}} \left(\lambda_{1,\nu}^{(2)}\right)^{\frac{d \beta }{2 + d \beta}}$
		& no \\
		& $\beta > 1$ &&& \\
    \midrule
		& $\alpha \in (0,2)$ &&& \\
   $\frac{C}{r^{d+\alpha}} \1_{\left\{r \leq 1\right\}} + \frac{C}{e^{\theta (r-1)^{\beta}}} \1_{\left\{r > 1\right\}}$ & $\theta > 0$
	  & $\frac{t}{(\log t)^{2/d}}$
		& $- \beta^{\frac{2}{d}} \, \left(\frac{\rho \omega_d}{d}\right)^{\frac{2}{d}}\lambda_{1,\nu}^{(2)}$
		& $- \beta^{\frac{2}{d}} \, \left(\frac{\rho \omega_d}{d}\right)^{\frac{2}{d}}\lambda_{1,\nu}^{(2)}$
		& \textbf{yes} \\
		& $\beta \in (0,1)$ &&& \\
		  \midrule
    & $\alpha \in (0,2)$ &&& \\
   $\frac{C}{r^{d+\alpha}} \1_{\left\{r \leq 1\right\}} + \frac{C}{e^{\theta (r-1)^{\beta}}} \1_{\left\{r > 1\right\}}$ & $\theta > 0$
	  & $\frac{t}{(\log t)^{2/d}}$
		& $- \left(\frac{\rho \omega_d}{d}\right)^{\frac{2}{d}}\lambda_{1,\nu}^{(2)}$
		& $- \left(\frac{\rho \omega_d}{d}\right)^{\frac{2}{d}}\lambda_{1,\nu}^{(2)}$
		& \textbf{yes} \\
		& $\beta \geq 1$ &&& \\
		  \midrule
    & &&& \\
   $\frac{C}{r^{d+\alpha}} \1_{\left\{r \leq 1\right\}}$ & $\alpha \in (0,2)$
	  & $\frac{t}{(\log t)^{2/d}}$
		& $- \left(\frac{\rho \omega_d}{d}\right)^{\frac{2}{d}}\lambda_{1,\nu}^{(2)}$
		& $- \left(\frac{\rho \omega_d}{d}\right)^{\frac{2}{d}}\lambda_{1,\nu}^{(2)}$
		& \textbf{yes} \\
		& &&& \\
		  \midrule
     	& &&& \\
   Brownian motion &
	  & $\frac{t}{(\log t)^{2/d}}$
		& $- \left(\frac{\rho \omega_d}{d}\right)^{\frac{2}{d}}\lambda_1^{BM}$
		& $- \left(\frac{\rho \omega_d}{d}\right)^{\frac{2}{d}}\lambda_1^{BM}$
		& \textbf{yes} \\
		& &&& \\
    \bottomrule
  \end{tabular}
	\smallskip
 \caption{Rate functions $\eta(t)$ and bounds for $\liminf_{t \to \infty} \frac{\log u^{\omega}(t,x)}{\eta(t)}$ and $\limsup_{t \to \infty} \frac{\log u^{\omega}(t,x)}{\eta(t)}$ for specific isotropic L\'evy processes. First six examples are pure jump processes with L\'evy-Khintchine exponents as in \eqref{eq:Lchexp} with $A \equiv 0$ and $\nu(dx)=\nu(|x|)dx$, where $\nu(r)$ are subsequent profiles given in the first column. Here $\lambda_{1,\nu}^{(\alpha)}$ and $\lambda_{1,\nu}^{(2)}$ denote the principal eigenvalues for the given stable process (in the first line) and diffusions determined by Gaussian matrices as in \eqref{eq:lambda_for_A_nu} (in the next five lines), killed on leaving the ball $B(0,1)$. We compare these examples with the case of Brownian motion which is included in the last line. In the last column we indicate for which processes the convergence follows.
}
\label{tab:tab}

\end{sidewaystable}

(2) If the L\'{e}vy measure decays stretched exponentially or faster (one must necessarily have $\alpha=2$ in this case), then the annealed rate is
$t^{d/(d+2)},$ while the quenched rate is bigger and equal to $t/(\log t)^{2/d}.$ This is the same rate as that for the Brownian motion. In this case, we not only identify the quenched rate, but also often obtain the limit (Theorem \ref{thm:exp} and Examples \ref{ex:exp}, \ref{ex:ex_irr} (1)). This case covers many examples of processes that are of interests in mathematical physics and in technical sciences, including the relativistic $\alpha-$stable process and some tempered stable processes \cite{bib:CMS90, bib:KL14}.

(3) We also consider {a class of processes with L\'{e}vy measures that have} intermediate decay: slower than {stretched} exponential, but
still faster than polynomial (Theorem \ref{thm:exp_log}). The annealed rate is {\it perforce} equal to $t^{d/(d+2)}$, but the quenched rate
obtained is $t^{(\beta d)/(2+\beta d)},$ $\beta $ being a parameter
specific to the process.

It is seen from this picture that the two interesting regime changes ('transitions') in the growth properties of the quenched rates occur. The first one can be observed when the intensity of large jumps of the processes varies from polynomial to higher order, in the sense that the quenched rate becomes faster than the annealed rate (i.e. it is no longer consistent with the annealed rate and becomes heavily dependent on the decay of the intensity of large jumps of the process). The second transition occurs as the intensity of large jumps becomes stretched exponential or faster. In this case, the long jumps intensity-driven quenched rate takes the form $t/(\log t)^{2/d}$, which is the quickest possible one, obtained also for Brownian motion. It is worth to point out that similar large jumps intensity-dependent transition in the ground state fall-off properties of the nonlocal Schr\"odinger operators has been recently identified in \cite{bib:KL15}

The verification of the assumptions of {our general Theorems \ref{thm:upper} and \ref{thm:lower_bound} for various types of L\'evy measures (i.e. in each of the situations (1)-(3) above)} requires a separate analysis. The applicability of our results essentially depends on the verifiability of the assumption \textbf{(U)} preceding Theorem \ref{thm:upper}. It asserts the existence of the profile function $F(r)$ that dominates the tail $\pr_0(|X_t|>r)$ for large $r$. This profile plays a crucial role in determining the quenched rate and therefore, in applications, it is a key initial step to establish it as precisely as possible. It does not come as a surprise that such a profile should be determined by the tail of the corresponding L\'evy measure. When \eqref{eq:psi-okura} holds with $\alpha\in(0,2)$, then the corresponding profiles $F(r)$ are derived by using the general estimates for the tails of the supremum functional obtained in \cite{bib:Pru}. When $\alpha = 2$, the problem is more complicated and it requires an application of the sharp estimates of the transition probability densities that are available in the literature. For L\'evy measures with stretched exponential and lighter tails, we apply directly the results of \cite{bib:CKK} while for those with polynomial and intermediate tails we use the estimates obtained recently in \cite{bib:KS14} (Lemmas \ref{lem:density_estimate}-\ref{lem:tran_dens_ub}). The case of jump processes with non-degenerate Gaussian components is discussed separately in Proposition \ref{prop:gauss_est}. Another key step in application of our general lower bound was to find a possibly sharpest lower estimate for the Dirichlet heat kernels of the large box which leads to sufficiently precise lower bound of the function $G$ defined in \eqref{eq:G-def}. For processes with L\'evy measures whose tails decay at infinity not faster than exponentially this is established in Proposition \ref{prop:dirichlet-by-nu}. The cases with lighter tails require an application of more specialized estimates obtained in \cite{bib:KiKi}.

At the end of the Introduction, let's say a few words about how the
general theorems Th. \ref{thm:upper}, Th. \ref{thm:lower_bound} are obtained.
{To the best of our knowledge the quenched asymptotics for L\'evy processes with jumps has not been studied before. }
In the literature {concerning the Brownian motion}, one finds two methods:
 Sznitman's paper {\cite{bib:Szn-ptrf93}} estimates $u^\omega(t,x)$ directly, using his `enlargement of obstacles' technique for the more difficult upper bound (similar method was used on the Sierpi\'{n}ski gasket in \cite{bib:kpp-ptrf});
 Fukushima \cite{bib:Fuk} gives elegant
arguments for deriving both the upper and the lower quenched bound from respective upper and lower bounds at zero for the integrated density of states of the {corresponding Schr\"odinger operator} (being closely related to the annealed upper and lower bounds) - this is done by means of the Dirichlet-Neumann bracketing for the
Laplace operator.
In our work, we are able to {find a counterpart of Fukushima's method for L\'evy processes with jumps} to obtain the upper bounds.  As the Dirichlet-Neumann bracketing {seems not to be available in the nonlocal case}, we had to use a different approach for the lower
bound. The lower estimate of $u^\omega(t,x)$ we provide is proven
directly, without using any properties of the annealed limits.

\section{Preliminaries}
\subsection{L\'{e}vy processes}\label{subsec:levy-proc}

Recall that $X=\pro X$ is assumed to be a symmetric jump L\'evy process in $\R^d$, $d \geq 1$, with L\'evy-Khintchine exponent $\psi$ as in \eqref{eq:Lchexp}. We will always assume that $X$ is strong Feller and
\begin{align} \label{eq:int}
\text{there exists \ $t_0>0$ \ for which \ ${\rm e}^{-t_0 \psi(\cdot)} \in L^1(\R^d)$.}
\end{align}
Note that the strong Feller property is equivalent to the existence
of measurable transition densities\linebreak $p(t,x,y)=p(t,y-x)$ for the process (see e.g. \cite[Th. 27.7]{bib:Sat}), while \eqref{eq:int} guarantees that\linebreak
$\sup_{x \in \R^d} p(t,x) = p(t,0) \leq p(t_0,0)  < \infty$, for any $t \geq t_0$.

Consequently, $X$ is  strong Markov with respect to its natural filtration and has a modification with c\`adl\`ag paths.
The c\`adl\`ag property will be assumed throughout the paper.
For more details on L\'evy processes we refer to \cite{bib:Sat,bib:J,bib:JSch,bib:App}.

The generator $L$ of the process $\pro X$ is a nonlocal pseudodifferential operator uniquely determined by its Fourier transform
\begin{align} \label{def:gen}
\widehat{L f}(\xi) = - \psi(\xi) \widehat{f}(\xi), \quad \xi \in \R^d, \quad f \in \mathcal D(L),
\end{align}
where $\mathcal D(L)=\left\{f \in L^2(\R^d): \psi \widehat f \in L^2(\R^d) \right\}$.
It is a negative-definite self-adjoint operator with a core $C_0^{\infty}(\R^d)$
such that  for $f \in C_0^{\infty}(\R^d)$
$$
L f(x) = \sum_{i,j=1}^d a_{ij} \frac{\partial^2 f}{\partial x_j \partial x_i} (x) + \int_{\R^d\setminus \{0\}}\left(f(x+z) - f(x) - z \cdot \nabla f(x) 1_{\{|z| \leq 1\}}(z)\right) \nu(dz), \quad x \in \R^d.
$$
The corresponding Dirichlet form $(\cE, D(\cE))$ can be defined by
\begin{align} \label{def:qform}
\cE(f,g) = \int_{\R^d} \psi(\xi) \widehat{f}(\xi) \overline{\widehat{g}(\xi)} d\xi, \quad f, g \in D(\cE),
\end{align}
with $D(\cE)=\left\{f \in L^2(\R^d): \int_{\R^d} \psi(\xi) |\widehat{f} (\xi)|^2 d\xi < \infty \right\}$.
It holds that $\cE(f,g)=(-Lf,g)$, for $f \in \mathcal D(L)$ and $g\in\mathcal D(\mathcal E).$

The transition densities $p^U(t,x,y)$ of the process killed upon exiting an open, bounded set $U \subset \R^d$ are given by the Dynkin-Hunt formula
\begin{align}
\label{eq:HuntF}
p^U(t,x,y)=p(t,x,y) - \ex_x\left[ \tau_U < t; p(t-\tau_U, X_{\tau_U},y)\right], \quad x , y \in U.
\end{align}
Here and thereafter, $\tau_U = \inf\left\{t \geq 0: X_t \notin U\right\}$ denotes the  exit time of the process from the set $U$.
The $L^2-$semigroup of operators with kernel $p^U(\cdot,\cdot,\cdot),$
also called the Dirichlet semigroup, will be denoted by $\big\{P^U_t: t\geq 0\big\}.$
Since $U$ is bounded, the operators $P^U_t$ are trace-class (consequently, compact) and admit a complete set of positive eigenvalues
\[\lambda_1(U)<\lambda_2(U)\leq \lambda_3(U)\ldots \to\infty.\]
Sometimes, to specify which process we are working with, these eigenvalues will be denoted by $\lambda_i^\psi(U),$ where $\psi$ is the L\'evy exponent of $\pro X.$ In the special case of symmetric $\alpha$-stable processes, $\alpha\in(0,2],$ its corresponding Dirichlet form will be denoted by $(\mathcal E^{(\alpha)},\mathcal D(\mathcal E^{(\alpha)})),$ and the eigenvalues of the Dirichlet semigroup -- by $\lambda^{(\alpha)}_i(U).$ For the standard Brownian motion running at twice the speed, we will use the notation $(\mathcal E^{BM},\mathcal D(\mathcal E^{BM}))$ and $\lambda^{BM}_i(U),$ respectively.

 \subsection{Poisson potentials}
 The process $X$ will be subject to interaction with a nonnegative,
 random Poissonian potential $V^\omega.$
 To properly set the assumptions, recall that the Kato class relative to $X,$ $\mathcal K^X,$ consists of those measurable functions $V:\mathbb R^d\to\mathbb R$ for which
\begin{align} \label{eq:Kato-def}
\lim_{t\to 0}\sup_{x\in \mathbb R^d}\int_0^t \mathbf E_x|V(X_s)|\,{\rm d} s=0,
\end{align}
 and the local Kato class $\mathcal K_{\rm loc}^X$ -- of functions $V$ such that for every ball $B=B(x,r)$ the function $V\cdot\mathbf 1_{B}\in\mathcal K^X.$ We always have $L_{\rm loc}^\infty(\R^d) \subset \mathcal K_{\rm loc}^X\subset L^1_{\rm loc}(\mathbb R^d).$
The condition defining the Kato class can be reformulated in terms of the kernel $p(t,x)$ restricted to small $t$ and small $x$: it is shown in
\cite[Corollary 1.3]{bib:GrzSz} that \eqref{eq:Kato-def} is equivalent to
\begin{align} \label{eq:Kato_new}
\lim_{t \to 0^{+}} \sup_{x \in \R^d} \int_0^t \int_{B(x,t)} p(s,x-y)|V(y)|\, {\rm d}y {\rm d}s = 0.
\end{align}
Sharp estimates of $p(t,x)$ that are available in the literature (see e.g. \cite{bib:KS14, bib:KS13, BGR13, bib:CKK}) often allow to find more explicit form of \eqref{eq:Kato_new}.

Further, let $\mathcal N$ be a Poisson point process on $\mathbb R^d$, with intensity $\rho\, {\rm d}x,$ $\rho>0,$  defined on some probability space $(\Omega, \mathcal M, \mathbb Q),$
and let $W:\mathbb R^d\to \mathbb R_+$ be a {$\mathcal K_{\rm loc}^X$ -- function satisfying
\begin{align} \label{eq:int_W}
\sup_{x \in B(0,R)} W(x - \cdot) \in L^1(B(0,2R)^c), \quad \text{for every \ $R>1$.}
\end{align}
Then define
\begin{align} \label{eq:pot}
V^{\omega}(x) = \int_{\R^d} W(x-y) \mu^{\omega}({\rm d}y),
\end{align}
where $\mu^{\omega}$ is the random counting measure on $\mathbb R^d$ corresponding to the Poisson point process $\mathcal N.$
For such profiles $W,$ the potential $V^\omega(\cdot)$ belongs $\mathbb Q-$almost surely to $\mathcal K_{\rm loc}^X.$ This can be directly justified by following the argument in \cite[Proposition 2.1]{bib:Kal-Pie-SPA}, where it has been proven for the subordinate Brownian motions on the Sierpi\'{n}ski gasket. 
One can check that when the profile $W$ is continuous, or when it is a nonincreasing function of the Euclidean distance, then the condition \eqref{eq:int_W} is satisfied under the assumption $W \in L^1(\R^d)$. Starting from Section \ref{sec:lower} we will be interested in the Poissonian potentials with finite-range (compactly supported) profiles $W$, for which \eqref{eq:int_W} holds automatically. By the range of a profile $W$ we mean $a:=\inf\left\{r > 0: W(x)=0 \ \mbox{for Lebesgue-almost all $x \in B(0,r)^c$ }\right\}$.

\subsection{Random semigroups and the integrated density of states}
Suppose that $W:\mathbb R^d\to\mathbb R_+$ is a profile function
belonging to {$\mathcal K_{\rm loc}^X$ for which \eqref{eq:int_W} holds.} As indicated above, $V^\omega$ given by (\ref{eq:pot}) belongs to $\mathcal K_{\rm loc}^X,$ $\mathbb Q-$almost surely.
{Therefore we can legitimately define the random Feynman-Kac semigroups $\big\{P_t^{V^\omega}: t\geq 0\big\}$ and $\big\{P_t^{U,V^\omega}: t\geq 0\big\}$ related to the 'free' process and the process killed on exiting an open, bounded and nonempty set $U \subset \mathbb R^d$. They consist of operators
$$P_t^{V^\omega}f(x)= \mathbf E_x\left[f(X_t){\rm e}^{-\int_0^t V^\omega(X_s)\,{\rm d}s}\right], \quad f \in L^2(\R^d), \ t >0,$$
and
$$P_t^{U,V^\omega}f(x)= \mathbf E_x\left[f(X_t){\rm e}^{-\int_0^t V^\omega(X_s)\,{\rm d}s}\mathbf 1_{\{\tau_U>t\}}\right], \quad f \in L^2(U), \ t >0,$$
and admit the measurable, strictly positive, bounded and symmetric kernels $p^{V^\omega}(t,x,y)$ and $p^{U,V^\omega}(t,x,y)$, respectively. It is known that $\mathbb Q-$a.s. the semigroup $\big\{P_t^{V^\omega}: t\geq 0\big\}$ coincides with the semigroup generated by the operator $-H^{\omega}$, where $H^{\omega} = -L + V^{\omega}$ is the \emph{random nonlocal Schr\"odinger operator} based on the  generator $L$ of the process $X$, with Poissonian potential $V^{\omega}$ \cite{DC, bib:CL}. The semigroup $\big\{P_t^{U,V^\omega}: t\geq 0\big\}$ corresponds then to the random nonlocal Schr\"odinger operator $H^{\omega}_U$ with exterior Dirichlet  conditions on $U$. The operators $P_t^{U,V^\omega}$ are Hilbert-Schmidt,} so that the spectrum of the operator {$H^{\omega}_U$} is
$\mathbb Q-$a.s. discrete:
\[\lambda_1(U,V^\omega)<\lambda_2(U,V^\omega)\leq \lambda_3(U,V^\omega)\ldots\to \infty.\]
Again, we will single out  the case of $\alpha-$stable processes and denote the respective eigenvalues by $\lambda_i^{(\alpha)}(U,V^\omega).$
Similarly, $P^V_t$ and $P^{U,V}_t$ will denote operators relative
to nonrandom potentials $0\leq V\in\mathcal K^X_{\rm loc}.$

Consider now the process killed on exiting the boxes $U=U_R=(-R,R)^d,$
 and the random empirical measures on $\mathbb R_+,$ based on the spectra the generators of such processes, normalized by the volume:
 \begin{equation}\label{eq:ids}
 \ell_R^\omega:=\frac{1}{(2R)^d}\sum_{n=1}^\infty \delta_{\lambda_n(U_R,V^\omega)}.\end{equation}

From the maximal ergodic theorem it follows that $\mathbb Q-$a.s. the measures $\ell^\omega_R$ are vaguely convergent as $R\to\infty$ to a nonrandom measure $\ell$ on $\mathbb R_+,$ called the {\em integrated density of states} (see e.g. \cite[p. 635]{bib:Okura}). The cumulative distribution function of the measure
$\ell$ will be denoted by $N^D(\lambda).$ The superscript $D$ indicates that we are dealing with the Dirichlet {exterior} conditions (as opposed to the {Neumann conditions, which are not pursued in this paper}).

\subsection{Notation}

We say that the function $g$ is asymptotically equivalent to the function $f$ at infinity, which is denoted by $g \approx f$, when
$\lim_{x \to \infty} f(x)/g(x) = 1$. Likewise, when we say $f\asymp g,$ we mean that there {exists a constant $C \in [1,\infty),$}  such that $\frac{1}{C}\,g(x)\leq f(x)\leq C g(x)$  for all relevant arguments $x$ (the range will be clear from the context).
For an open set $U\subset \mathbb R^d,$  $C_c^\infty(U)$ stands for $C^\infty$--functions with compact support inside $U.$
$B(x,R)$ denotes the open Euclidean ball with center $x$ ans radius $R>0.$
We also say that a measurable function $W:\mathbb R^d\to\mathbb R_+$ is not identically zero, if $|\{x\in \mathbb R^d: W(x)>0\}|>0$ (by $|U|$ we denote the Lebesgue measure of the set $U$).
Important constants are denoted with upper case letters $C,K,Q$, possibly with subscripts.
Technical constants are numbered
within each proof separately as $c_1,c_2,...$.

\section{The upper bound}\label{sec:upper}

\subsection{Preliminary estimates}
We start with two preliminary results. First, a lemma, proven for nonrandom potentials. Recall that the constant $t_0$ comes from the assumption \eqref{eq:int}.

\begin{lemma} \label{lem:lem1}
Let $(X_t)_{t\geq 0}$ be a symmetric {strong Feller} L\'evy process with  L\'evy-Khintchine exponent $\psi$ as in {\eqref{eq:Lchexp} and \eqref{eq:int}}, and let $0 \leq V \in \cK_{\loc}^X$. Then there exists {a constant $C_1 = C_1(X,d)$} such that for any open,
nonempty set $U \subset \R^d$ one has
\begin{equation}\label{eq:l1}
P^V_t \1(x) \leq C_1 {|U|}^{1/2} \, {\rm e}^{-\lambda_1(U,V)(t-t_0/2)}+ \pr_x[\tau_U \leq t], \quad x \in U, \ \ t > t_0/2.
\end{equation}
\end{lemma}

\begin{proof} The proof  goes along  standard arguments. Let $U, x,$ and $t$ be as in the assumptions. We have
\begin{eqnarray*}
P_t^V \1(x) &\leq & \ex_x\left[{\rm e}^{-\int_0^t V(X_s)\,{\rm d}s};\tau_U>t\right] + \pr_x[\tau_U\leq t]\\
&=& P_t^{U,V} \1_U(x)+ \pr_x[\tau_U\leq t]
\end{eqnarray*}
and for any $R>0$
\begin{eqnarray*}
P_t^{U,V}\1_U(x) &\leq & P_t^{U,V}\mathbf 1_{B(x,R)}(x)+ P_t^{U,V}\mathbf 1_{B(x,R)^c}(x)\\
&\leq & P_{t_0/2}^{U,V}P_{t-t_0/2}^{U,V}\1_{B(x,R)}(x)+\pr_x[X_t\notin B(x,R)].
\end{eqnarray*}
Further,
\begin{eqnarray*}
P_{t_0/2}^{U,V}P_{t-t_0/2}^{U,V}\1_{B(x,R)} & = & \left\langle p^{U,V}(t_0/2,x,\cdot), P_{t-t_0/2}^{U,V}\1_{B(x,R)}(\cdot)\right\rangle_{L^2(U)}\\
&\leq & \|p^{U,V}(t_0/2,x,\cdot)\|_{L^2(U)}\|P_{t-t_0/2}^{U,V}\1_{B(x,R)}\|_{L^2(U)}
\\
&\leq & (p^{U,V}(t_0,x,x))^{1/2}{\rm e}^{-\lambda_1(U,V)(t-t_0/2)}\|\1_{B(x,R)}\|_{L^2(U)}
\\ & \leq& c (p(t_0,0))^{1/2}{\rm e}^{-\lambda_1(U,V)(t-t_0/2)} \sqrt{|U\cap B(x,R)|}.
\end{eqnarray*}
Collecting these estimates we obtain
\begin{equation*}
P^V_t \1(x) \leq C_1\sqrt{|U\cap B(x,R)|} \, {\rm e}^{-\lambda_1(U,V)(t-t_0/2)}+ \pr_x[\tau_U \leq t] + \pr_0[|X_t| \geq R], \quad x \in U, \ \ t > t_0/2.
\end{equation*}
To get \eqref{eq:l1}, it is enough to take the limit $R \to \infty$.
\end{proof}

In the random setting, we will need the following lemma on the mean number of eigenvalues not exceeding a given level $\lambda>0$.

\begin{lemma} \label{eq:number_eig}
Let $X$ be a symmetric {strong Feller} L\'evy process with characteristic exponent $\psi$ satisfying \eqref{eq:Lchexp} {and \eqref{eq:int},} and let $V^{\omega}$ be a Poissonian potential defined in \eqref{eq:pot}. For $n \in \Z_+$ let $D_n = (-2^n, 2^n)^d.$ Then for every $\lambda >0$ we have
\begin{equation}\label{eq:ids-lambda-1}
2^d \, \ex_{\qpr} \left[\# \left\{k \in \Z_+: \lambda_k(D_n,V^{\omega}) \leq \lambda \right\}\right] \leq \ex_{\qpr} \left[\# \left\{k \in \Z_+: \lambda_k(D_{n+1},V^{\omega}) \leq \lambda \right\}\right], \quad n \in \Z_+.
\end{equation}
Consequently, for any box $D_n$ as above and any  $\lambda>0$  one has
\begin{equation}\label{eq:ids-lambda-2}
\frac{1}{|D_n|}\ex_{\qpr} \left[\# \left\{k \in \Z_+: \lambda_k(D_{{n}},V^{\omega}) \leq \lambda \right\}\right]\leq N^D(\lambda). \end{equation}
\end{lemma}

\begin{proof}
Let $n \in \Z_{+}$. Denote by $\left\{D^i_n\right\}_{i=1}^{2^d}$ the collection of $2^d$ disjoint open boxes of the form $\mathbf x+{(0,2^{n+1})^d}$, $\mathbf x \in {2^{n+1}} \Z^d$ such that $\bigcup_{i=1}^{2^d} D^i_n \subset D_{n+1}$ and $|\bigcup_{i=1}^{2^d} D^i_n| = |D_{n+1}| = 2^{(n+2)d}$.

By using standard min-max formulas for  eigenvalues (see, e.g., \cite[Section 12.1]{bib:Schm}), one can check that
\begin{align} \label{eq:ineq_number}
 \sum_{i=1}^{2^d} \#\{k \in \Z_{+}:\lambda_k(D^i_n, V^{\omega}) \leq \lambda\} \leq \#\{k \in \Z_{+}:\lambda_k(D_{n+1}, V^{\omega}) \leq \lambda\}.
\end{align}
Moreover, the space homogeneity of the process together with the stationarity of the potential $V^{\omega}$ give
$$
\ex_{\qpr} \left[\#\{k \in \Z_{+}:\lambda_k( D^i_n,V^{\omega}) \leq \lambda \} \right] = \ex_{\qpr} \left[\#\{k \in \Z_{+}:\lambda_k(D_n,V^{\omega}) \leq \lambda \} \right], \quad i = 1,...,2^d.
$$
Taking the expected value $\ex_{\qpr}$ on both sides of \eqref{eq:ineq_number}, we immediately get \eqref{eq:ids-lambda-1}.
\end{proof}

\subsection{A general upper bound}
\noindent
We first introduce an auxiliary function through which we determine the typical asymptotic profile for
the quenched asymptotics of the function $u^{\omega}(t,x)$.

For every $\alpha \in (0,2]$, $\kappa >0,$ and a nonincreasing function $F:[1,\infty) \to [0,\infty)$ with $\lim_{r \to \infty} F(r)=0$ we define the function $f_{F,\alpha,\kappa}:[1,\infty) \to [0,\infty)$ by
\begin{align} \label{eq:basicf}
f_{F,\alpha, \kappa}(r) = \left(\big(r \wedge |\log (1 \wedge F(r))| \big) + \frac{d}{2} \log r \right) \left(\frac{d \log r}{\kappa}\right)^{\frac{\alpha}{d}}.
\end{align}
One can directly see that for any fixed $\alpha \in (0,2]$, $\kappa >0,$ and a given function $F(r)$ we have
$$
f_{F,\alpha, \kappa}(1) =  0, \quad \lim_{r \to \infty} f_{F,\alpha,\kappa}(r) = \infty, \quad  \text{moreover $f_{F,\alpha, \kappa}(r)$ is strictly increasing in $r$}.
$$
In particular, the inverse function
\begin{equation}\label{eq:ha_te}
h_{F,\alpha, \kappa}:=[f_{F,\alpha, \kappa}]^{-1} : [0,\infty) \to [1,\infty)
\end{equation}
is well defined. It is strictly increasing and satisfies
$$
\lim_{t \to \infty} h_{F,\alpha, \kappa}(t) = \infty.
$$
When the parameters $\alpha$ and $\kappa$ will be fixed, they will be
dropped. Observe that the function $h_{F,\alpha,\kappa}(t)$ satisfies:
\begin{equation}\label{eq:eq-ha}
t\left(\frac{\kappa}{d\log h_{F,\alpha,\kappa}(t)}\right)^{\frac{\alpha}{d}}=\Big( h_{F,\alpha,\kappa}(t) \wedge \left|\log \left( 1 \wedge F(h_{F,\alpha,\kappa}(t))\right) \right|\Big)+\frac{d}{2}\,\log h_{F,\alpha,\kappa}(t).
\end{equation}
The function $h_{F,\alpha,\kappa}(t)$ will play a central role in determining the rate of decay of the functionals considered.

\smallskip

In what follows we
will work under the following  regularity assumption \textbf{(U)} on the process $X$.
In the condition below, the constant $t_0$ comes from the assumption \eqref{eq:int}.

\bigskip

\begin{itemize}
\item[\textbf{(U)}] There are constants {$C_2 >0$, $C_3 \geq 2$}, {$\gamma > 0$, $r_0 \geq 1$, $t_1 \geq t_0 \vee 1$
and a nonincreasing function $F:[1, \infty) \to [0,\infty)$ such that $\lim_{r \to \infty} F(r) = 0,$} for which
\begin{align*}
\pr_0(|X_t| \geq r) \leq {C_2} \, t^{\gamma} \, \left(F(r) \vee e^{-r} \right),
\quad r \geq r_0 \vee {C_3 t}, \quad t \geq t_1.
\end{align*}
\end{itemize}

\bigskip

The next theorem is our main result in this section.

\begin{theorem} \label{thm:upper}
Let $X$ be a symmetric {strong Feller} L\'evy process
with characteristic exponent $\psi$ satisfying \eqref{eq:Lchexp}
 {and \eqref{eq:int}}
  such that the assumption \textbf{(U)} holds.
  Let $V^{\omega}$ be a Poissonian potential defined in \eqref{eq:pot}.
If there exist $\alpha \in (0,2]$ and $\kappa_0 >0$ such that
\begin{align} \label{eq:IDS_upper}
\limsup_{\lambda \to 0^{+}} \, \lambda^{d/\alpha} \, \log N^D(\lambda) \leq - \kappa_0,
\end{align}
then for every fixed $x \in \R^d$ one has
\begin{equation}\label{eq:statement-upper}
\limsup_{t \to \infty} \frac{\log u^{\omega}(t,x) - \frac{d}{2} \log h_{F,\alpha,\kappa_0}(t)}{g(t)} \leq -\left(\frac{\kappa_0}{d}\right)^{\alpha/d}, \quad \qpr-\text{a.s.},
\end{equation}
where the function $h_{F,\alpha,\kappa}$ is defined in \eqref{eq:ha_te} with $F$ given by \textbf{(U)} and $g(t):= t/(\log h_{F,\alpha,\kappa}(t))^{\alpha/d}$.
\end{theorem}

\begin{proof}
Fix $x \in \R^d$
and let $r_0$, $t_1$, $\gamma$, $\alpha$, $\kappa_0$ and $F$ be as in the assumptions.
Specifically, we may and do assume that $r_0 \geq 1$ is so large that $F(r) \leq 1$ for $r \geq r_0$.
We will write $h$ for $h_{F,\alpha,\kappa_0}$.
By Lemma \ref{lem:lem1}, for every $t \geq t_0/2$ and every open set
 $U\ni x$, we have
\begin{align}\label{eq:bd0}
u^{\omega}(t,x) = P_t^{V^{\omega}} \1 (x)
\leq C_1 \sqrt{|U|} {\rm e}^{-\lambda_1(U,V^{\omega})(t-t_0/2)}
+ \pr_x[\tau_U \leq t].
\end{align}
In particular, we can choose $U=U_{2R}= (-2R, 2R)^d,$ where $R$ is so
large that
\begin{align} \label{eq:restr}
R > |x| \vee r_0 \vee {C_3} t \quad \text{and} \quad t \geq t_1.
\end{align}
Now, since for this choice of $U$ we have $B(x,R)\subset U,$ from the L\'{e}vy inequality and assumption {\bf (U)}\linebreak we obtain:
\begin{eqnarray} \label{eq:bd1}
\pr_x[\tau_U \leq t] &\leq& \pr_x[\tau_{B(x,R)} \leq t] = \pr_0[\tau_{B(0,R)} \leq t]
=\pr_0[\sup_{s \in [0,t]} |X_s| \geq R]  \nonumber\\ &\leq& 2 \pr_0[|X_t| \geq R]\leq 2C_2 t^\gamma {\rm e}^{-(R\wedge |\log (F(R))|)}.
\end{eqnarray}
We now estimate $\lambda_1({U_{2R}},V^{\omega})$ for large $R$.
Inequality \eqref{eq:ids-lambda-2} from Lemma \ref{eq:number_eig}
holds for dyadic boxes $D_n$ and reads:
\[
\frac{1}{2^{(n+1)d}} \ex_{\qpr} \left[\# \left\{k \geq 1: \lambda_k(D_n, V^{\omega}) \leq \lambda \right\}\right] \leq N^D(\lambda), \quad \lambda > 0, \ \ n \in \Z_+.
\]
Running the argument from  \cite[(2.3)-(2.6)]{bib:Fuk}
 with $\phi(r) = \kappa_0 r^{d/\alpha}$ and the sequence
  $t_n=2^n$, from the assumption \eqref{eq:IDS_upper},
   we get that  for every $\varepsilon \in (0,1),$  $\qpr-$almost surely we can find $R_{\varepsilon}=R_{\varepsilon}(\omega) > 1$ such that for every $R \geq R_{\varepsilon}$
\begin{align} \label{eq:rev}
\lambda_1({U_{2R}},V^{\omega}) \geq (1-\varepsilon)
\left(\frac{\kappa_0}{d \, \log R} \right)^{\alpha/d}.
\end{align}
Piecing together
\eqref{eq:bd0}, \eqref{eq:bd1}, and \eqref{eq:rev}
we get that
for every $\varepsilon \in (0,1),$ $\mathbb Q-$a.s.
there exists $R_{\varepsilon} > 1$
 such that for all $t$ and $R$ satisfying $R \geq R_{\varepsilon}$
 and \eqref{eq:restr} one has
\begin{equation*}
u^\omega(t,x)\leq C_1 (4R)^{d/2} {\rm e}^{-(1-\varepsilon)(t-t_0/2)
\left(\frac{\kappa_0}{d\log R}\right)^{\alpha/d}} + 2C_2t^\gamma{\rm e}^{-(R\wedge
|\log(F(R))|)},
\end{equation*}
and further,
\begin{align} \label{eq:b3}
\log u^{\omega} (t,x) & \leq -\min \left\{(1-\varepsilon)
\left(t-\frac{t_0}{2}\right) \left(\frac{\kappa_0}{d \, \log R}
\right)^{\alpha/d} - \frac{d}{2} \log R \, , \, R \, \wedge
 \, \left|\log \Big(F(R)\Big)\right| \right\} \\ &
  \ \ \ + \gamma \log t + c_1,  \nonumber
\end{align}
with an absolute constant $c_1 \geq 0$.

Let now $h(t)$ be given by (\ref{eq:ha_te}).
As $h(t)\to \infty$ when $t\to\infty$, $\mathbb Q-$a.s. there exists $t_2\geq t_1$ large enough so that for every $t \geq t_2$ the condition
\eqref{eq:restr} holds with $R=h(t) \vee {C_3 t},$ and moreover $R\geq R_\varepsilon(\omega).$  Thus we may substitute in \eqref{eq:b3} the value
\begin{equation}\label{eq:er-t}R=R(t):= h(t) \vee {C_3 t}.
\end{equation}
Next, from the definition of $h(t),$   (\ref{eq:eq-ha}), and  the monotonicity of $f_{F,\alpha, \kappa_0}$, we see that
\begin{align} \label{eq:est_h}
\big(h(t) \vee {C_3 t} \big) \, \wedge \, \left|\log \Big(1 \wedge F(h(t) \vee {C_3 t}) \Big)\right|
 \geq t \left(\frac{\kappa_0}{d \, \log \big(h(t) \vee {C_3 t}\big)} \right)^{\frac{\alpha}{d}} - \frac{d}{2} \log \big(h(t) \vee {C_3 t}\big)
\end{align}
with  equality when $h(t) \geq {C_3 t}$. We finally obtain that for all $t > t_2$ we have
\begin{align*}
\log u^{\omega}(t,x) & \leq  - \left((1-\varepsilon)(t-t_0/2) \left(\frac{\kappa_0}{d \, \log \big(h(t) \vee {C_3 t}\big)} \right)^{\alpha/d} - \frac{d}{2} \log \big(h(t) \vee {C_3 t}\big)\right) + \gamma \log t + {c_1} \\
&\leq - \left((1-\varepsilon)(t-t_0/2) \left(\frac{\kappa_0}{d \, \log \big(h(t) \vee {C_3 t}\big)} \right)^{\alpha/d} - \frac{d}{2} \log h(t) \right) + \left(\gamma + \frac{d}{2}\right) \log t + {c_2},
\end{align*}
with absolute constants  $c_1, c_2 \geq 0$, for $\qpr$--almost all $\omega$.

 To complete the proof,
it remains to show that {
\begin{align} \label{eq:b4}
\log t = o(g(t)) \qquad \text{as \ $t \to \infty,$}
\end{align}
where $g(t):= t/(\log h(t))^{\alpha/d}$.
This is obvious when $h(t) \leq {C_3 t}$, and when $h(t) \geq {C_3 t}$, then from \eqref{eq:eq-ha} we have
\[t \left(\frac{\kappa_0}{d \, \log h(t)} \right)^{\alpha/d} - \frac{d}{2} \log h(t)>0, \quad  \mbox{so that} \quad h(t)\leq e^{c_3 t^{\frac{d}{d+\alpha}}},\]
for some $c_3 > 0$ and \eqref{eq:b4} follows.}

We conclude that $\mathbb Q-$almost surely
$$
\limsup_{t \to \infty} \frac{\log u^{\omega}(t,x) - \frac{d}{2} \log h(t)}{g(t)} \leq -(1-\varepsilon) \left(\frac{\kappa_0}{d}\right)^{\alpha/d},  \ \ \text{with} \ \ g(t) =  \frac{t}{\left( \log h(t) \right)^{\alpha/d}}.
$$
Letting $\varepsilon \to 0$ through rational numbers, we get (\ref{eq:statement-upper}). The proof is complete.
\end{proof}

The next corollary will eventually enable us to
obtain, for certain processes, the existence of $\lim_{t\to\infty}\frac{\log u^\omega(t,x)}{g(t)}.$

\begin{corollary} \label{cor:upper}
Let the assumptions of  Theorem \ref{thm:upper} above be satisfied. More specific, let {\bf (U)} hold with a function $F$.
If there exists $Q_1 \in (0,\infty]$ such that
$$
\liminf_{r \to \infty}\frac {|\log F(r)|}{\log r} \geq {Q_1}
$$
then
\begin{equation}\label{eq:to_zero}
\limsup_{t \to \infty} \frac{\log h_{F,\alpha,\kappa_0}(t)}{g(t)}
 \leq \frac{2}{2 {Q_1} + d} \left(\frac{\kappa_0}{d}\right)^{\alpha/d}
\end{equation}
and, consequently, for every fixed $x \in \R^d$, one has
\begin{equation}\label{eq:statement-upper_A}
\limsup_{t \to \infty} \frac{\log u^{\omega}(t,x)}{g(t)}
 \leq -\left(\frac{\kappa_0}{d}\right)^{\alpha/d}
 \left(1-\frac{d}{d+2{Q_1}}\right) , \quad \qpr-\text{a.s.}
\end{equation}
In particular, when $\lim_{r\to\infty}\frac{|\log F(r)|}{\log r}=\infty,$ then
\begin{equation}\label{eq:statement-upper_infty}
\limsup_{t \to \infty} \frac{\log u^{\omega}(t,x)}{g(t)} \leq -\left(\frac{\kappa_0}{d}\right)^{\alpha/d} , \quad \qpr-\text{a.s.}
\end{equation}

\end{corollary}

\begin{proof}
The assumptions give that for any $0 < {\widetilde Q_1} < {Q_1}$ there exists $r_0$ such that
for $r>r_0$
\[f_{F,\alpha,\kappa_0}(r)\geq \left( r\wedge
 ({\widetilde Q_1} \log r)+\frac{d}{2}\,\log r\right)\left(\frac{d}{\kappa_0}\,\log r\right)^{\alpha/d},\]
which is equivalent to saying that for sufficiently large $r$
\[f_{F,\alpha,\kappa_0}(r)\geq \left({\widetilde Q_1} + \frac{d}{2}\right)\left(\frac{d}{\kappa_0}\right)^{\alpha/d}(\log r)^{1+\alpha/d}\]
or, for sufficiently large $t,$
\[h_{F,\alpha,\kappa}(t)\leq {\rm e}^{\left(\frac{2}{2 {\widetilde Q_1} + d} \left(\frac{\kappa_0}{d}\right)^{\alpha/d} t \right)^{\frac{d}{d+\alpha}}},\]
and further
\[0\leq \frac{\log h_{F,\alpha,\kappa_0}(t)}{g(t)}\leq
 \frac{(\log h_{F,\alpha,\kappa_0}(t))^{\frac{\alpha+d}{d}}}{t}
 \leq \frac{2}{2 {\widetilde Q_1} + d} \left(\frac{\kappa_0}{d}\right)^{\alpha/d}.\]
This means that
\[0\leq\limsup_{t\to\infty}\frac{\log h_{F,\alpha,\kappa_0}(t)}{g(t)}
\leq \frac{2}{2 {\widetilde Q_1} + d} \left(\frac{\kappa_0}{d}\right)^{\alpha/d}.\]
Letting {$\widetilde Q_1 \nearrow Q_1$} we get \eqref{eq:to_zero}. Statements \eqref{eq:statement-upper_A} and \eqref{eq:statement-upper_infty} follow directly from \eqref{eq:statement-upper}.
\end{proof}

\section{The lower bound for regularly distributed L\'evy measures}\label{sec:lower}

As indicated in the Introduction, the argument
deriving the quenched asymptotic lower bound directly from the lower asymptotics of the IDS seems to be not obvious in the non-diffusion case.
Instead, we estimate $u^\omega(t,x)$ directly.
In this part (similarly as in \cite{bib:Szn-ptrf93, bib:Fuk}), we require the potential profile $W$ to be bounded and compactly supported.
As usual in problems of this kind, we first prove that $\mathbb Q-$almost surely there
exist sufficiently large regions without potential interaction,
then we force the process to go to this region and then stay there for a long enough time. This behaviour will be described analytically.

\subsection{Typical potential configuration}

Let $\varepsilon>0$ be given.
For a given number $r>0,$ let $M^\varepsilon(r)$ be defined by
\begin{equation}\label{eq:m-M}
M^\varepsilon(r):= \left(\frac{d}{\omega_d \rho (1+\varepsilon)}\right)^{\frac{2}{d}+2} r^{-2d-2} {{\rm e}^{\frac{\omega_d\rho(1+\varepsilon)}{d}r^d}},
\end{equation}
where $\omega_d$ denotes the volume of the unit ball in $\mathbb R^d$.
We have a lemma.

\begin{lemma}\label{lem:empty_box} Let $\varepsilon>0,$ $a>0,$  and  let $r$ and $M^\varepsilon(r)$ be related by (\ref{eq:m-M}). Then
$\mathbb Q-$almost surely, there exists $r_0>0$ such that for
$r>r_0$ the set $(-M^\varepsilon(r),M^\varepsilon(r))^d\setminus (-r,r)^d$ contains a ball  with radius $r$ whose $a-$neighbourhood is free of Poisson points.
\end{lemma}

\noindent{\bf Proof.} Assume first that $r=m\in\mathbb Z$. Then
\begin{equation}\label{eq:no_Poisson}
\mathbb Q[\mbox{the ball $B(0,{m+a})$ contains no Poisson points}] = 
{\rm e}^{-\omega_d\rho(m+a)}.
\end{equation}
The {equality} (\ref{eq:no_Poisson}) is true for any translate of $B(0,{m+a}).$ Inside the box $(-M^{\varepsilon/2}(m), M^{\varepsilon/2}(m))^d,$ one can pack $a_m:=\left[\frac{2 M^{\varepsilon/2}(m)}{ 2(m+a)}\right]^d$
disjoint boxes of size  $2m$ that are $(2a)-$separated. Open balls of radius $(m+a),$  `concentric' with those boxes are disjoint. As the realizations of the cloud over disjoint sets are independent random variables, the probability that each such (small) ball contains at least one Poisson point (denote this event by $\mathcal A_m^0$) equals to
\[
\mathbb Q[\mathcal A_m^0]= (1-{\rm e}^{-\omega_d\rho{(m+a)^d}})^{a_m}.\]
We would like to produce a ball with radius $(m+a)$ that is both: free of Poisson points and separated from zero, so that we exclude from our considerations the boxes whose closure might contain zero, at most $2^d$ of them.
Let $\mathcal A_m$ be the event that `every small ball from $(-M^{\varepsilon/2}(m), M^{\varepsilon/2}(m))^d\setminus (-{m-a},{m+a})^d,$ arising as above, contains a Poisson point', then
\[
\mathbb Q[\mathcal A_m]= \left(1-{\rm e}^{-\omega_d\rho{(m+a)^d}}\right)^{a_m-2^d}.\]
Using an elementary inequality $(1-x)\leq {\rm e}^{-x}$ we can write
\[\mathbb Q[\mathcal A_m] \leq\exp\left\{-{\rm e}^{-\omega_d\rho {(m+a)^d}}\left(\frac{M^{\varepsilon/2}(m)}{(m+a)}\right)^d(1-o(1))\right\},\;\;\; \mbox{for } m\to\infty,\]
and the expression in the exponent is equal to (recall $r=m$)
\[-(1-o(1))\,{\rm e}^{\omega_d\rho \frac{\varepsilon}{2} \,m^d{(1-o(1))}},\quad m\to\infty\]
so that
\[\sum_{m=1}^\infty \mathbb Q[\mathcal A_m]<\infty.\]
From the Borel-Cantelli lemma we get that $\mathbb Q-$almost surely
there is a number $m_0>0$ such that for $m>m_0$ the set $(-M^{\varepsilon/2}(m), M^{\varepsilon/2}(m))^d\setminus (-{m-a},{m+a})^d$ contains a ball of radius ${(m+a)}$ free of Poisson points. {In particular, if $r>m_0$ is a real number, then} we can find an empty ball of size {$[r]+1 + a$} included in the big box $(-M^{\varepsilon/2}([ r]+1), M^{\varepsilon/2}([ r]+1))^d$ and separated from zero. Since
\[ M^{\varepsilon/2}([r]+1)\leq M^\varepsilon(r)\;\;\; \mbox{for sufficiently large }r,\]
the lemma follows. \hfill $\Box$

\medskip

We also quote Lemma 3.2  from \cite{bib:Fuk}  (we have $\alpha=1$ in present case).
\begin{lemma}\label{lem:potential} Suppose that the profile function $W$ is compactly supported {and bounded}. Then
$\mathbb Q-$almost-surely, for sufficiently large $R$ one has
\begin{equation}\label{eq:pot-est}
\sup_{x\in (-R,R)^d} V_\omega(x) \leq 3d \log R.
\end{equation}
\end{lemma}

\subsection{A general lower bound}
\noindent
Let $R > R_0 >0$ be given and let $p^{U_R}(t,x,y)$ be the Dirichlet kernel of our process $\pro X$ in the box $U_R:=(-R,R)^d$. To begin with, we introduce the following notation:
\begin{equation}\label{eq:G-def}
G(R_0,R):=\inf_{ \frac{R_0}{2} \leq |y| \leq \frac{R}{2}} p^{U_R}(1,0,y).
\end{equation}
Also, recall that $\lambda_1^{{(\alpha)}}(B(0,1))$ is the principal Dirichlet eigenvalue of the symmetric $\alpha$--stable process defined by (\ref{eq:psi-stable}) in the unit ball $B(0,1)$ and $\omega_d$ is the volume of this ball.

We now present our main theorem in this section.
\begin{theorem}\label{thm:lower_bound}
Let $X$ be a symmetric strong Feller L\'evy process with {characteristic} exponent $\psi$ satisfying \eqref{eq:Lchexp} { and \eqref{eq:int}}.
Moreover, suppose that there exist $\alpha\in (0,2]$ and {$K>0$} such that
\begin{equation}\label{eq:lower-1}
\lambda_1^{\psi}(B(0,R)) \leq {K}  R^{-\alpha}\lambda^{{(\alpha)}}_1(B(0,1))+o(R^{-\alpha}),\quad R\to\infty,
\end{equation}
and that $V^{\omega}$ is a Poissonian potential defined in \eqref{eq:pot} with bounded profile $W$ of finite range.
Then for any $x\in\mathbb R^d$, $\kappa, R_0>0$ and any nonincreasing function $F:[0, \infty) \to [0,\infty)$ such that $\lim_{r \to \infty} F(r) = 0$, $\mathbb Q-$almost surely,
\[\liminf_{t\to\infty}\frac{\log u^\omega(t,x)+3d\log h_{F,\alpha,\kappa}(t)
+\left|\log {j_{R_0, F, \alpha, \kappa}(t)}\right|}{g(t)}
\geq - K \left(\frac{\omega_d\rho}{d}\right)^{\alpha/d} \lambda_1^{{(\alpha)}}(B(0,1)),\]
where  $h_{F,\alpha,\kappa}(t)$ was  defined by \eqref{eq:ha_te},  $g(t)= \frac{t}{(\log h_{F,\alpha,\kappa}(t))^{\alpha/d}}$, and ${j_{R_0, F, \alpha, \kappa}(t)}:=G\left(R_0,\frac{2\sqrt{d} h_{F,\alpha,\kappa}(t)}{(\log h_{F,\alpha,\kappa}(t))^{\frac{2}{d}+2}}\right)$ (the function ${j_{R_0, F, \alpha, \kappa}(t)}$ is well-defined for large $t$'s).
\end{theorem}

\begin{proof}
 For simplicity, we run the proof for $x=0$ only; for a general $x\in \mathbb R^d$ the proof is identical. Let $\kappa>0$ and {$R_0>0$} be given. As in the proof of Theorem \ref{thm:upper}, we will write $h$ for $h_{F,\alpha,\kappa}$.
Let $\varepsilon>0$ be given and let $a$ be the range of the potential profile $W$,
then for $t>0$ let $m(t)$ and $M(t)(=M^\varepsilon(m(t)))$ be related by
(\ref{eq:m-M}). The potential range $a$ is fixed so it does not enter the notation.  For the time being  we require only that $ m(t)\to\infty$ when $t\to\infty.$  Eventually, the number $m(t)$ will be chosen of order $h(t)$ from
(\ref{eq:ha_te}), but in such a manner  that $M^\varepsilon(m)$ will bear no $\varepsilon-$dependence.

 Pick $\omega $ outside the exceptional sets from
Lemmas \ref{lem:empty_box} and  \ref{lem:potential}. Let ${B_t}$ be the open ball
of radius $m(t)$ whose $a-$neighbourhood contains no Poisson points, obtained
from the statement of Lemma \ref{lem:empty_box}.
As there is no interaction with the potential inside this ball, we have
 that $p^{B_t,V^{\omega}}(\cdot,\cdot,\cdot)=p^{{B_t}}(\cdot,\cdot,\cdot),$ and consequently $\lambda_1^{\psi}({B_t},V^\omega)=\lambda_1^\psi({B_t})$ (recall that
$\lambda^{\psi}(U,{V^{\omega}})$ and $\lambda_1^\psi(U)$ denote the principal Dirichlet eigenvalue of the process in $U$ under the influence of the
potential $V^\omega,$ or  without potential interaction, respectively).

 Let $\phi$ be the normalized, positive Dirichlet $L^2-$eigenfuntion,
 supported in ${B_t},$  corresponding to
this principal eigenvalue.
For sufficiently large $t$ we have the following chain of inequalities:
\begin{eqnarray*}
u^\omega(t,0)&=& \int_{\mathbb R^d} p^{{V^{\omega}}}(t,0,x) \,{\rm d}x = \int_{\mathbb R^d}\int_{\mathbb R^d} p^{{V^{\omega}}}(1,0,y)p^{{V^{\omega}}}(t-1,y,x)\,{\rm d}y\,{\rm d}x \\
& \geq & \int_{{B_t}}\int_{B_t} p^{{V^{\omega}}}(1,0,y)p^{{V^{\omega}}}(t-1,y,x)\,{\rm d}y\,{\rm d}x
\\
&\geq & \left(\inf_{y\in B_t} p^{{V^{\omega}}}(1,0,y) \right) \int _{B_t}  \left(\int_{B_t} p^{{V^{\omega}}}(t-1,y,x)\,\frac{\phi(x)}{\|\phi\|_\infty} \,{\rm d}x\right){\rm d}y\\
&\geq & \left(\inf_{y\in B_t} p^{{V^{\omega}}}(1,0,y) \right) \int _{B_t}  \left(\int_{B_t} p^{{V^{\omega}},B_t}(t-1,y,x)\,\frac{\phi(x)}{\|\phi\|_\infty} \,{\rm d}x\right){\rm d}y\\
&= & \frac{1}{\|\phi\|_\infty}\left(\inf_{y\in B_t} p^{{V^{\omega}}}(1,0,y) \right) \int _{B_t} {\rm e}^{-\lambda_1^\psi({V^{\omega}},{B_t})(t-1)}\phi(y)\,{\rm d}y\\
&\geq & \frac{1}{\|\phi\|_\infty^2}\left(\inf_{y\in B_t} p^{{V^{\omega}}}(1,0,y) \right) {\rm e}^{-\lambda_1^\psi({B_t})(t-1)}\int_{B_t} \phi(y)^2\,{\rm d}y \\
&\geq& \frac{1}{\|\phi\|_\infty^2} \,{\rm e}^{-t\lambda_1^\psi({B_t}) }\left(\inf_{y\in B_t} p^{{V^{\omega}}}(1,0,y) \right).
\end{eqnarray*}
From the translation invariance of the process and assumption (\ref{eq:lower-1})  we see that
\begin{equation}\label{eq:u10}
\lambda^\psi_1({B_t})\leq {K}m(t)^{-\alpha}\lambda_1^{(\alpha)}(B(0,1))+o(m(t)^{-\alpha}),\quad\mbox{ as }t\to\infty.
\end{equation}
Also, it is classical to see that $\|\phi\|_\infty\leq {c_1} \lambda^\psi_1({B_t})^{d/(2\alpha)},$ which
from \eqref{eq:u10} can be estimated as ${c_2} m(t)^{-d/2},$
so that $\|\phi\|_\infty^2 m(t)^d\leq {c_3}.$
The chain of inequalities continues as
\begin{equation}\label{eq:a}\geq\,
{c_4} \left(\inf_{y\in B_t} p^{{V^{\omega}}}(1,0,y)\right) \, {{\rm e}^{-t\lambda^\psi_1(B_t)}} m(t)^d.
\end{equation}

To estimate the infimum of the kernel $p^{{V^{\omega}}}(1,0,y)$ for $y\in B_t,$ we take $J_t=(-2 \sqrt{d} M(t),2 \sqrt{d} M(t))^d.$ For $y\in B_t$ one has $y \in (-M(t),M(t))^d \setminus (-m(t),m(t))^d$ so that for sufficiently large $t$ one has
$R_0\leq y\leq \sqrt d M(t)$.
Using {\eqref{eq:pot-est} and} \eqref{eq:G-def} we can write:
{
\begin{align}\label{eq:point1}
p^{{V^{\omega}}}(1,0,y)&\geq p^{J_t,{V^{\omega}}}(1,0,y)\geq{\rm e}^{-3d \log (2 \sqrt{d} M(t))}p^{J_t}(1,0,y) \\
&\geq {c_5}{\rm e}^{{-3d\log(2\sqrt{d}M(t))}} G(R_0,2\sqrt{d}M(t)).\nonumber
\end{align}}

Inserting these estimates inside (\ref{eq:a}) and using {\eqref{eq:u10} again,}
we obtain that $\mathbb Q-$a.s., for sufficiently large $t$:
\begin{align*}\label{eq:b}
& u^\omega (t,0) \\ & \geq {c_{6}}{\exp}\left\{-3d\log (2\sqrt{d}M(t)) -|\log G(R_0, 2\sqrt{d}M(t))| -{K} t m(t)^{-\alpha}(\lambda_1^{(\alpha)}(B(0,1))+o(1)) + d\log m(t)\right\}. \nonumber
\end{align*}
At this point we declare the scale $m(t).$
Recall that all this reasoning is performed for a fixed number $\varepsilon>0.$
Set $m(t)=m_\varepsilon(t)$  to be the solution of the equation (unique for large $t$)
\begin{equation}\label{eq:m_epsilon}
\omega_d\rho (1+\varepsilon)(m_\varepsilon(t))^d=d\log h(t),
\end{equation}
where $h(t)$ was given by (\ref{eq:ha_te}).
Consequently, using \eqref{eq:m-M},
\[M(t)=M^\varepsilon(m_\varepsilon(t))= \frac{h(t)}{(\log h(t))^{{\frac{2}{d}}+2}}.\]
It follows
\begin{align}\label{eq:c}
\log u^\omega(t,0)  & + {3d\log \left(\frac{{2\sqrt{d}h(t)}}{(\log h(t))^{\frac{2}{d}+2}}\right)}  + \left|\log G\left({R_0, \frac{2\sqrt{d}h(t)}{(\log h(t))^{\frac{2}{d}+2}}}\right)\right| \\
&\geq  -K tm_\varepsilon(t)^{-\alpha}(\lambda_1^{(\alpha)}(B(0,1))+o(1))+ d\log m_\varepsilon(t) + O(1),\quad t\to\infty. \nonumber
\end{align}
Since $h(t)\to\infty$ as $t\to\infty$, it is immediate that $m_\varepsilon(t)\to\infty$ when $t\to\infty.$ 
Due to (\ref{eq:m_epsilon}) we get
 \[\lim_{t\to\infty}\left( -\frac{t(m_\varepsilon(t))^{-\alpha}}{g(t)}\right)=  -\left(\frac{\omega_d\rho(1+\varepsilon)}{d}\right)^{\alpha/d}\quad \mbox{and } \quad d \log m_\varepsilon(t) = \log \log h(t) + O(1) \ \ \ \mbox{when \ $t \to \infty$.}  \]
Further, from the relation  (\ref{eq:eq-ha}) defining $h,$ we see that
 {\[g(t)=\frac{t}{(\log h(t))^{\alpha/d}} \geq c_7
 \,\log h(t). \]}
 Consequently, for sufficiently large $t$ we get
 \[0\leq\frac{\log\log h(t)}{g(t)}\to 0\quad \text{when} \quad t\to\infty.\]
These properties give that, $\mathbb Q-$almost surely,
 \[\liminf_{t\to\infty}\frac{\log u^\omega(t,0)+ {3d\log h(t)}+{\left|\log {j_{R_{0},F,\alpha,\kappa}(t)}\right|}}{g(t)}\geq - K\left(\frac{\omega_d\rho(1+\varepsilon)}{d}\right)^{\alpha/d}\lambda_1^{(\alpha)}(B(0,1)).\]
Letting $\varepsilon\to 0$ through rationals gives the statement.
\end{proof}

The next corollary gives a direct lower bound for $\liminf \frac{\log u^\omega(t,x)}{g(t)},$ similar to that in Corollary \ref{cor:upper}.

\begin{corollary} \label{cor:lower}
Let the assumptions of the above theorem be satisfied. In particular, let $F$ and $G$ be the monotone
functions appearing in its statement.
If there exist {$Q_1 \in (0,\infty]$ and $Q_2 \in [0,\infty)$} such that
\begin{equation}\label{eq:A_lower}
\liminf_{r \to \infty}\frac {|\log F(r)|}{\log (r)} \geq Q_1
\end{equation}
and
\begin{equation}\label{eq:B_lower}
\limsup_{r \to \infty}\frac {\left|\log G{\left(R_0, \frac{2\sqrt{d}r}{(\log r)^{\frac{2}{d}+2}}\right)}\right|}{(r \wedge |\log F(r)|)+ (d/2)\log r} \leq Q_2,
\end{equation}
then
\begin{equation}\label{eq:to_zero_A_B}
\limsup_{t \to \infty} \frac{\log h_{F,\alpha,\kappa}(t)}{g(t)} \leq \frac{2}{2 Q_1+ d} \left(\frac{\kappa}{d}\right)^{\alpha/d} \quad \text{and} \quad
\limsup_{t \to \infty} \frac{\left|\log {j_{R_0,F,\alpha,\kappa}(t)}\right|}{g(t)} \leq Q_2 \left(\frac{\kappa}{d}\right)^{\alpha/d}
\end{equation}
and, consequently, for every fixed $x \in \R^d$, one has
\begin{equation}\label{eq:statement-upper:A}
\liminf_{t \to \infty} \frac{\log u^{\omega}(t,x)}{g(t)} \geq -{K}\left(\frac{\omega_d\rho}{d}\right)^{\alpha/d}\lambda_1^{(\alpha)}(B(0,1)) -\left(\frac{6d}{d+2Q_1} + Q_2\right)\,\left(\frac{\kappa}{d}\right)^{\alpha/d}, \quad \qpr-\text{a.s.}
\end{equation}
In particular, when the assumptions (\ref{eq:A_lower}) and (\ref{eq:B_lower})  hold with  $Q_1=\infty$ and $Q_2=0,$
then
\begin{equation}\label{eq:A-infty-B-zero}
\liminf_{t \to \infty} \frac{\log u^{\omega}(t,x)}{g(t)} \geq -{K}\left(\frac{\omega_d\rho}{d}\right)^{\alpha/d}
\lambda_1^{(\alpha)}(B(0,1)).
\end{equation}
\end{corollary}

\begin{proof}
The first bound in \eqref{eq:to_zero_A_B} follows from \eqref{eq:A_lower} exactly by the same argument as in Corollary \ref{cor:upper}.
To prove the second bound in \eqref{eq:to_zero_A_B} we write
$$
\frac{\left|\log j_{R_0}(t)\right|}{g(t)}
= \frac{\left|\log G\left({R_0,\frac{2\sqrt{d}h_{F,\alpha,\kappa}(t)}{(\log h_{F,\alpha,\kappa}(t))^{\frac{2}{d}+2}}}\right)\right|}{(d/\kappa)^{\alpha/d}\big((h_{F,\alpha,\kappa}(t) \wedge |\log F(h_{F,\alpha,\kappa}(t))|)+ (d/2)\log h_{F,\alpha,\kappa}(t)\big)},
$$ The desired bound {immediately follows from \eqref{eq:eq-ha}} once we recall that $h(t)\to\infty$ when $t\to\infty.$
\end{proof}

\section{Discussion of specific cases}\label{sec:specific}
We will apply the general results of previous sections to some particular processes, for which the assumptions of Theorems \ref{thm:upper} and \ref{thm:lower_bound} hold true.
Throughout this section we will work under the assumption that the L\'{e}vy-Khinchine exponent $\psi$ is close to the characteristic exponent of a symmetric $\alpha-$stable process near the origin.
More precisely, we assume the following condition.

\begin{itemize}
\item[\bf{(C)}] One has
$$
\psi(\xi)=\psi^{(\alpha)}(\xi)+ o(|\xi|^\alpha), \quad \text{when \ $|\xi| \to 0$,}
$$
where $\psi^{(\alpha)}$ is given by \eqref{eq:psi-stable}-\eqref{eq:psi-diff} for some $\alpha \in (0,2]$ and satisfies $\inf_{|\xi|=1} \psi^{(\alpha)}(\xi) >0$.
\end{itemize}

\smallskip
\noindent
{We will also assume some regularity on the behaviour of $\psi$ at infinity, a kind of Hartman-Wintner condition:
\begin{align} \label{eq:psi_inf}
 \frac{\psi(\xi)}{(\log|\xi|)^2}\to \infty \quad \text{as } |\xi|\to\infty.
\end{align}
Observe that under this assumption $e^{-t \sqrt{\psi(\cdot)}} \in L^1(\R^d)$, for every
$t>0$. In particular, $e^{-t \psi(\cdot)} \in L^1(\R^d)$, $t >0$, and \eqref{eq:int} holds for every $t_0>0$.
}

Under these assumptions, in the paper {\cite{bib:Okura81}} the annealed asymptotics of $u^\omega(t,x)$ was proven, and also {in \cite{bib:Okura}} the behaviour of the integrated density of states, $N^D(\lambda),$ was
established. We have the following.

\begin{theorem}\label{th:Okura} \cite[Theorem 6.2]{bib:Okura}
Let {$X$ be a symmetric} L\'{e}vy process whose characteristic exponent $\psi$ satisfies {\bf (C)} and \eqref{eq:psi_inf}, and let $V^{\omega}$ be a Poissonian potential defined in \eqref{eq:pot} with nonnegative and nonidentically zero $\cK^X_{\loc}$-class profile $W$ satisfying the condition $W(x)=o(|x|^{-d-\alpha})$, $|x|\to\infty.$  Then
\begin{equation}
\label{eq:Okura}
\lim_{\lambda\to 0^+} \lambda^{d/\alpha}\log N^D(\lambda)=-\rho (\lambda_{(\alpha)})^{d/\alpha}.
\end{equation}
The constant $\lambda_{(\alpha)}$ is given by the variational formula
\begin{equation}\label{eq:lambda}
\lambda_{(\alpha)}=\inf_G\lambda_1^{(\alpha)}(G),
\end{equation}
where the infimum is taken over all open sets $G\subset\mathbb R^d$ of unit Lebesgue measure.
\end{theorem}

{Theorem 6.2 in \cite{bib:Okura} has been proven for continuous profiles $W$, but its proof also applies to the local Kato-class case.}

Moreover, it follows from the Faber-Krahn isoperimetric inequality (see, e.g. \cite[Lemma 3.13]{bib:Don-Var} and \cite[Theorem 3.5]{bib:blp}) that when the process $X^{(\alpha)}$ is isotropic, then the infimum in
(\ref{eq:lambda}) is attained on the ball of radius $r_d=\omega_d^{-1/d}$ ($\omega_d$ is the volume of the unit ball)
and is equal to $\omega_d^{\alpha/d}\lambda_1^{(\alpha)}(B(0,1)).$

\smallskip

Theorem \ref{th:Okura} above states that {\bf (C)} and \eqref{eq:psi_inf} are sufficient conditions for the validity of \eqref{eq:IDS_upper}, which is the main assumption of Theorem \ref{thm:upper}. We now show that when {\bf (C)} holds, then also the quasi-scaling of principal eigenvalues needed in Theorem \ref{thm:lower_bound} holds true. The following proposition takes care of that.

\begin{proposition}\label{prop:conditionA}
Suppose that $X$ is a L\'{e}vy process such that condition {\bf (C)} is satisfied with certain $\alpha\in(0,2]$.
Then also (\ref{eq:lower-1}) holds true, with $\alpha$ the same as that in {\bf (C)} and {any $K>1$. More precisely, for any fixed $K >1$ one has
$$
\lambda_1^{\psi}(B(0,R))  \leq  K R^{-\alpha}\lambda^{{(\alpha)}}_1(B(0,1))+o(R^{-\alpha}),\quad R\to\infty.
$$
}
\end{proposition}

\begin{proof} Suppose that {\bf (C)} is true; let {$\psi^{(\alpha)}$ be the L\'evy-Khinchine exponent of the symmetric $\alpha-$stable process $X^{(\alpha)}$} appearing in this condition. Denote $\overline{\psi}(\xi)=\psi(\xi)-\psi^{(\alpha)}(\xi),$ so that for fixed $R>0$
\[\psi\left(\frac{\xi}{R}\right)= \psi^{(\alpha)}\left(\frac{\xi}{R}\right)+ \overline{\psi}\left(\frac{\xi}{R}\right).\]
For any given $u\in C_c^\infty(\mathbb R^d)$ let $u_R(x)=R^{-d/2}u(\frac{x}{R}).$
Then
\begin{eqnarray}\label{eq:forms-scaling}
\mathcal E(u_R,u_R) &= &
 \int_{\mathbb R^d} \psi(\xi)|\widehat{u_R}(\xi)|^2\,d\xi = \int_{\mathbb R^d} R^{d}\psi(\xi)|\widehat{u}(R\xi)|^2\,{\rm d}\xi\nonumber
\\
&=& \int_{\mathbb R^d}
\psi\left(\frac{\xi}{R}\right)|\widehat{u}(\xi)|^2\,{\rm d}\xi
\nonumber\\
&=&
\int_{\mathbb R^d}
\psi^{(\alpha)}\left(\frac{\xi}{R}\right)|\widehat{u}(\xi)|^2\,{\rm d}\xi+\int_{\mathbb R^d}\overline{\psi}\left(\frac{\xi}{R}\right)|\widehat u(\xi)|^2\,{\rm d}\xi\nonumber\\
&=&\mathcal E^{(\alpha)}(u_R,u_R)+\int_{\mathbb R^d}\overline{\psi}\left(\frac{\xi}{R}\right)|\widehat u(\xi)|^2\,{\rm d}\xi\nonumber\\
&=& R^{-\alpha}\mathcal E^{(\alpha)}(u,u)+ \int_{\mathbb R^d}\overline{\psi}\left(\frac{\xi}{R}\right)|\widehat u(\xi)|^2\,{\rm d}\xi.
\end{eqnarray}

We have assumed that $\overline{\psi}(\xi)=o(|\xi|^\alpha)$ when $|\xi|\to 0.$ Therefore, since both
${0 \leq} \psi(\xi),\psi^{(\alpha)}(\xi)\leq c|\xi|^2$ when $|\xi|>1,$ we get
 that there is $c_1>0$ for which
\[{|\overline{\psi}(\xi)|}\leq c_1(|\xi|^\alpha+|\xi|^2),\quad \xi\in\mathbb R^d,\]
so that
\[{\left|R^{\alpha}\overline{\psi}\left(\frac{\xi}{R}\right)\right|}\leq c_1 (|\xi|^\alpha+|\xi|^2),\quad \xi\in\mathbb R^d,\; R\geq 1.\]
{Moreover, $R^{\alpha} \overline{\psi}(\xi/R)\to 0$ as $R\to\infty$,
for every fixed $\xi \in\mathbb R^d$.}
Since {$u \in C_c^\infty(\mathbb R^d)$}, the integral $\int_{\mathbb R^d}|\xi|^2|\widehat u(\xi)|^2{\rm d}\xi$ is finite, and from the dominated convergence theorem we obtain that
\begin{equation}\label{eq:ac}
\lim_{R\to\infty} R^{\alpha}\int_{\mathbb R^d} \overline{\psi}\left(\frac{\xi}{R}\right)|\widehat{u}(\xi)|^2{\rm d}\xi =  0.
\end{equation}

{Let now $\phi^{(\alpha)}$ be the first eigenfunction of the generator of the process $X^{(\alpha)}$ killed outside $B(0,1)$. Clearly, $\phi^{(\alpha)} \in \mathcal D(\mathcal E^{(\alpha)})$ and $\lambda_1^{(\alpha)}(B(0,1)) = \mathcal E^{(\alpha)}(\phi^{(\alpha)},\phi^{(\alpha)})$. Also, let $\left\{\varphi_{\delta}\right\}_{\delta>0} \subset C_c^\infty(\mathbb R^d)$ be a family of mollifiers such that $\supp \varphi_{\delta} \subseteq B(0,\delta)$. Denote $\phi^{(\alpha)}_{\delta} = (\phi^{(\alpha)} * \varphi_{\delta})/\left\|\phi^{(\alpha)} * \varphi_{\delta}\right\|_{2}$. Then, for every $\delta>0$, $\phi^{(\alpha)}_{\delta} \subset C_c^\infty(\mathbb R^d)$, $||\phi^{(\alpha)}_{\delta}||_2 = 1$ and $\supp \phi^{(\alpha)}_{\delta} \subseteq B(0,1+\delta)$. Moreover, $\mathcal E^{(\alpha)}(\phi^{(\alpha)}_{\delta},\phi^{(\alpha)}_{\delta}) \to \mathcal E^{(\alpha)}(\phi^{(\alpha)},\phi^{(\alpha)}) = \lambda_1^{(\alpha)}(B(0,1))$ as $\delta \to 0^{+}$.
}

{Suppose that $K>1$ is fixed, and choose $\delta >0$ small enough so that
\[ \mathcal E^{(\alpha)}(\phi^{(\alpha)}_{\delta},\phi^{(\alpha)}_{\delta}) \leq K \, (1+\delta)^{-\alpha} \lambda_1^{(\alpha)}(B(0,1)). \]
Then, from \eqref{eq:forms-scaling} it follows that for every $R >1$
\begin{eqnarray}
\lambda^\psi_1(B(0,(1+\delta)R)) \leq
\mathcal E((\phi^{(\alpha)}_{\delta})_R,(\phi^{(\alpha)}_{\delta})_R) &=& \frac{1}{R^{\alpha}}\,\mathcal E^{(\alpha)}(\phi^{(\alpha)}_{\delta},\phi^{(\alpha)}_{\delta})+
\int_{\mathbb R^d}\overline{\psi}\left(\frac{\xi}{R}\right)|\widehat{\phi^{(\alpha)}_{\delta}}(\xi)|^2{\rm d}\xi
\nonumber\\
&\leq & \frac{K}{((1+\delta)R)^\alpha} \,\lambda_1^{(\alpha)}(B(0,1)) + \int_{\mathbb R^d}\overline{\psi}\left(\frac{\xi}{R}\right)|\widehat{\phi^{(\alpha)}_{\delta}}(\xi)|^2{\rm d}\xi \nonumber
\end{eqnarray}
(the first inequality follows by the standard variational formula for the principal eigenvalue). Finally, by substituting $\widetilde R = (1+\delta) R$, we get
\begin{align}\label{eq:ad}
\lambda^\psi_1(B(0,\widetilde R)) \leq
K \, \widetilde R^{-\alpha} \,\lambda_1^{(\alpha)}(B(0,1)) + \int_{\mathbb R^d}\overline{\psi}\left(\frac{(1+\delta)\xi}{\widetilde R}\right)|\widehat{\phi^{(\alpha)}_{\delta}}(\xi)|^2{\rm d}\xi.
\end{align}
}
The statement follows now from
(\ref{eq:ac}) and (\ref{eq:ad}).
\end{proof}

We now provide some reasonable and easy-to-check sufficient conditions under which the basic asymptotic assumption {\bf (C)}  holds true.

\begin{proposition}\label{prop:suff_for_Okura}
Let $X$ be a L\'{e}vy process determined by the L\'evy-Khintchine exponent $\psi$ as in \eqref{eq:Lchexp} with  Gaussian coefficient $A=(a_{ij})_{1 \leq i,j \leq d}$ and  L\'evy measure $\nu$. The following hold.
\begin{itemize}
\item[(i)] If  $\nu$ has second moment finite, i.e.
$$
\int_{\mathbb R^d} |z|^2\,\nu({\rm d}z)<\infty,
$$
then
$$
\psi(\xi)=\xi \cdot \widetilde A \xi + o(|\xi|^2) \quad \text{as} \ \ \quad |\xi| \to 0,
$$
where $\widetilde A = (\widetilde a_{ij})_{1 \leq i, j \leq d}$ with $\widetilde a_{ij} = a_{ij} + \frac{1}{2}\int_{\mathbb R^d}y_iy_j\nu({\rm d}y)$.
\item[(ii)] If there exist $\alpha \in (0,2)$ and a symmetric finite measure $n$ on the unit sphere $S^{d-1}$ such that
$$
r^{2} \int_{1 < |z| \leq 1/r} |z|^2 |\nu - \nu^{(\alpha)}|({\rm d}z)= o(r^\alpha) \quad \text{and} \quad  |\nu - \nu^{(\alpha)}|\big(B(0,1/r)^c\big)= o(r^\alpha)  \quad \text{as \ $r \to 0$},
$$
with $\nu^{(\alpha)}({\rm d}r {\rm d}\varphi) = n({\rm d} \varphi)r^{-1-\alpha} {\rm d}r$, then
$$
\psi(\xi)= \psi^{(\alpha)}(\xi)+o(|\xi|^{\alpha}) \quad \text{as} \ \ \quad |\xi| \to 0,
$$
where $\psi^{(\alpha)}(\xi) = \int_0^\infty\int_{S^{d-1}} (1-\cos(\xi\cdot r\varphi))\,n({\rm d}\varphi)\frac{{\rm d}r}{r^{\alpha+1}}$.
\end{itemize}
\end{proposition}

\begin{proof}
Knowing that $\psi(\xi)= \xi\cdot A\xi +
\int_{\mathbb R^d} (1-\cos(\xi\cdot y))\,\nu({\rm d}y),$ and writing down the Taylor expansion of the function {$\cos s$ at $0$} we get
\[\psi(\xi)=\xi\cdot A\xi+\int_{\mathbb R^d} (\xi\cdot y)^2 \left[\int_0^1(1-t)\cos(t\xi\cdot y)\,{\rm d}t\right]\nu({\rm d}y).\]
The first assertion follows from the dominated convergence theorem together with the finiteness of the second moment of $\nu.$

To prove the second assertion, we write
$$
\psi(\xi) = \psi^{(\alpha)}(\xi) + \xi \cdot A \xi + \int_{\mathbb R^d} (1-\cos(\xi \cdot y))\,(\nu - \nu^{(\alpha)})({\rm d} y).
$$
Since $0 \leq \xi \cdot A \xi \leq \left\|A\right\||\xi|^2$, we only need to show that the last member above is of order $o(|\xi|^{\alpha})$. We have
\begin{align*}
\left|\int_{\mathbb R^d} (1-\cos(\xi \cdot y))\,(\nu - \nu^{(\alpha)})({\rm d} y)\right| &
\leq \frac{1}{2} \int_{|y| \leq 1/|\xi|} (|\xi||y|)^2 |\nu - \nu^{(\alpha)}|({\rm d}y) + 2 \int_{|y| > 1/|\xi|} |\nu - \nu^{(\alpha)}|({\rm d}y),
\end{align*}
and the statement follows from the assumption.
\end{proof}

In what follows we will {often} use the following notation. If $X=(X_t)_{t \geq 0}$ is a symmetric L\'evy process with characteristic exponent $\psi$ as in \eqref{eq:Lchexp}, then we
write
$$
X = X^A + X^{\nu} \quad \text{and} \quad \psi(\xi)=\psi_A(\xi)+\psi_\nu(\xi),
$$
where $X^A=(X^A_t)_{t \geq 0}$ is the Gaussian part determined by the L\'evy-Khintchine exponent $\psi_A(\xi)=\xi\cdot A\xi$, and $X^{\nu}=(X^{\nu}_t)_{t \geq 0}$ is the jump part with the exponent $\psi_\nu(\xi)=\int_{\R^d\setminus \{0\}} (1-\cos(\xi \cdot z)) \nu({\rm d}z)$.

{The following fact on the tails of jump L\'evy processes with nondegenerate Gaussian component will also be needed below.} It states that one can add a sufficiently regular diffusion process to a purely jump L\'evy process without spoiling the assumption $\textbf{(U)}$.

\begin{proposition}\label{prop:gauss_est}
Let $X$ be a L\'{e}vy process determined by the L\'evy-Khintchine exponent $\psi$ as in \eqref{eq:Lchexp} with Gaussian coefficient $A=(a_{ij})_{1 \leq i,j \leq d}$ and  L\'evy measure $\nu$. Moreover, suppose that $\inf_{|\xi|=1} \xi \cdot A \xi > 0.$ If
the process $X^{\nu}$ satisfies the assumption $\textbf{(U)}$ with $\gamma > 0$, profile  $F$ and constants {${C_2}, r_0, t_1$,} then the entire process $X$ also satisfies a version of $\textbf{(U)}$. More precisely, there are constants ${\widetilde C_2 \geq C_2}$ and {$C_4\in (0,1]$} such that
$${
\pr_0(|X_t| \geq r) \leq \widetilde C_2\, t^{\gamma} \left(F(C_4r) \vee e^{-C_4r}\right), \qquad  r \geq 2(r_0 \vee 2t), \ \ t \geq t_1.}
$$
In particular, if $F(C_4r) \geq e^{-C_4 r}$ for $r \geq 2r_0$, then $X$ satisfies the assumption $\textbf{(U)}$ with $\widetilde C_2$, the same $\gamma$ and the profile $\widetilde F(r) = F(C_4r)$. If $F(C_4r) < e^{-C_4 r}$ for $r \geq 2r_0$, then the same is true with $\widetilde C_2$, $\gamma$ and $\widetilde F(r) = e^{-C_4r}$.
\end{proposition}

\begin{proof}
For $t>t_1$ and $r \geq 4t$, we may write
 \begin{eqnarray*}
 \pr_0(|X_t|\geq r)&=&  \pr_0\left(|X_t^\nu+X_t^A|\geq r\right)\\
 &\leq & \pr_0\left(|X_t^\nu|\geq\frac{r}{2}\right) + \pr_0\left(|X_t^A|\geq\frac{r}{2}\right).
 \end{eqnarray*}
 The first part is estimated using {\bf (U)}. To take care of the Gaussian part,
note that under the assumption $\inf_{|\xi|=1} \xi \cdot A \xi > 0$ the transition densities $p_A(t,x,y)$ of the corresponding diffusion process $X^A$ exist and enjoy the Gaussian upper estimates:
\[p_A(t,x,y)\leq \frac{c_{1}}{t^{d/2}}{\rm e}^{-c_{2} \frac{|x-y|^2}{t}}.\]
In particular, taking into account that $r>4t\geq 4t_1$,
 \begin{eqnarray*}
 \pr_0\left(|X_t^A|\geq\frac{r}{2}\right) &= & \int_{B(0,r/2)^c} p_A(t,0,y){\rm d}y \leq \int_{B(0,r/2)^c}\frac{c_1}{t^{d/2}} \,{\rm e}^{-c_2\frac{|y|^2}{t}}\,{\rm d}y
 \leq c_3 {\rm e}^{-c_2r},
 \end{eqnarray*}
so that
\begin{align*}
\pr_0(|X_t| \geq r)
			 & \leq  C_2 \, t^{\gamma} \left(F(r/2) \vee e^{-r/2} \right) + c_3 e^{-c_2 r} \\
			 & \leq c_4 \, t^{\gamma} \left(F(c_5 r) \vee e^{-c_5 r} \right).
			\end{align*}
for some positive constants $c_1-c_5$. The proof is complete.
\end{proof}

\smallskip

In the sequel, we will also need the following general lower estimate for the function $G$ defined in \eqref{eq:G-def}.
Recall that for every $R>0$ we have denoted $U_R=(-R,R)^d.$ Below we will also write $B(R)$ for $B(0,R)$.

\begin{proposition}\label{prop:dirichlet-by-nu}
Let $X$ be a L\'{e}vy process with L\'evy-Khinchine exponent $\psi$ given by
\eqref{eq:Lchexp} {and such that $e^{-t\psi(\cdot)} \in L^1(\R^d)$, $t>0$}. Suppose there exist {$C_5 > 0$} and $r_0 > 0$ such that for every Borel set $E \subset \R^d$
\begin{align} \label{eq:ass1}
\nu \big(E \cap \left\{x:0<|x|\leq 2r_0\right\}\big) \geq {C_5} |E \cap \left\{x:0<|x|\leq 2r_0\right\}|.
\end{align}
Then for every $t>0,$  $R > 4r_0,$  and $2r_0 < |y| < \frac{R}{2}$  one has
\begin{equation}\label{eq:dirichlet-by-f}
p^{U_R}(t,0,y)\geq \eta(t,r_0) \inf_{|z|\leq\frac{R}{2}+r_0}
\nu(B(z,r_0)).
\end{equation}
where
\begin{eqnarray}\label{eq:def_eta}
\eta(t,r)&=& \int_0^{{t}} \mathbf P_0(\tau_{B(r)}>s)\,{\Lambda}(t-s,r)\min\left(1, {\Lambda}(t-s,r)\big|B\big(\frac{{r}}{8}\big)\big|\right)\,{\rm d}s.
\end{eqnarray}
with
\[{\Lambda}(t,r)={C_5} \int_0^{t/2} \mathbf P_0(\tau_{B(r/8)}>u)\mathbf P_0(\tau_{B({r}/8)}>t/2-u)\,{\rm d}u.\]
In particular,
\begin{equation}\label{eq:G-by-nu}
G(4r_0,R)\geq \eta(1,r_0) \inf_{|y|\leq\frac{R}{2}+r_0}\nu(B(y,r_0)).
\end{equation}
\end{proposition}

\begin{proof}
Let $r_0> 0$  be as in the assumptions and let $R > 4r_0$ be fixed. One can check by using the strong Markov property that for any $2r_0 < |y| <R/2$ and any $t>0$
\begin{equation}\label{eq:G_1}
p^{U_R}(t,0,y)= \mathbf E_0[\tau_{B(r_0)}<t; p^{U_R}(t-\tau_{B(r_0)}, X_{\tau_{B(r_0)}},y)].
\end{equation}
Set $\nu(x,\cdot) := \nu(\cdot - x)$. Using \eqref{eq:G_1} and the Ikeda-Watanabe formula \cite[Theorem 1]{bib:IW}, we can write for such $y$ and $t$
\begin{eqnarray*}
p^{U_R}(t,0,y) & \geq & \mathbf E_0\left[X_{\tau_{B(r_0)}}\in B(y,r_0), \tau_{B(r_0)}< t;
p^{U_R}(t-\tau_{B(r_0)}, X_{\tau_{B(r_0)}}, y)\right]\\
&=& \int_0^{t} \int _{B(r_0)}p^{B(r_0)}(s,0,z)\int_{B(y,r_0)}p^{U_R}(t-s,w,y)\nu(z,{\rm d}w)\,{\rm d}z\,{\rm d}s\\
&\geq & \int_0^{t}\int_{B(r_0)} p^{B(r_0)}(s,0,z)\inf_{{w \in B(y,r_0)}} p^{{B(y,2r_0)}}(t-s,{y},{w})\nu(z, B(y,r_0))\, {\rm d} z\,{\rm d}s
\\
&{\geq}& \left[\int_0^t \mathbf P^0\left(\tau_{B(r_0)}>s\right) \, \inf_{|x|\leq r_0}p^{B(2r_0)}(t-s,0,x)\,{\rm d}s \right] \inf_{|z|\leq\frac{R}{2}+r_0}
\nu(B(z,r_0)).
\end{eqnarray*}
To complete the proof, we need to estimate the kernel $p^{B(2r_0)}(t',0,x)$ for every $t' \in (0,t]$ and $|x| \leq r_0$. Let first $r_0/2 < |x| \leq r_0$.
By following through with  the argument above and  using \eqref{eq:ass1}, we have
\begin{eqnarray} \label{eq:G_2}
p^{B(2r_0)}(t',0,x) & \geq & \mathbf E_0\left[X_{\tau_{B(r_0/4)}}\in B(x,r_0/4), \tau_{B(r_0/4)}< t';
p^{B(2r_0)}(t'-\tau_{B(r_0/4)}, X_{\tau_{B(r_0/4)}}, x)\right] \nonumber \\
&=& \int_0^{t'} \int _{B(r_0/4)}p^{B(r_0/4)}(u,0,z)\int_{B(x,r_0/4)}p^{B(2r_0)}(t'-u,w,x)\nu(z,{\rm d}w)\,{\rm d}z\,{\rm d}u \nonumber\\
&\geq & {C_5} \int_0^{t'} \pr_0\left(\tau_{B(r_0/4)}>u\right) \, \int_{B(x,r_0/4)} p^{B(x,r_0/4)}(t'-u,x,w) {\rm d}w \,{\rm d}u\\
& = & {C_5}\int_0^{t'} \pr_0\left(\tau_{B({r_0}/4)}>u\right) \, \pr_0\left(\tau_{B(r_0/4)}>t'-u\right) \,{\rm d}u . \nonumber
\end{eqnarray}
In the case when $|x| \leq r_0/2$, use first the Chapman-Kolmogorov identity and then \eqref{eq:G_2} to get
\begin{align*}
& p^{B(2r_0)} (t',0,x) \geq \int_{r_0/2 < |z| \leq r_0} p^{B(2r_0)}(t'/2,0,z) p^{B(2r_0)}(t'/2,z,x) \,{\rm d}z \\
& \geq  {C_5} \left[\int_0^{t'/2} \pr_0\left(\tau_{B(r_0/4)}>u\right) \, \pr_0\left(\tau_{B(r_0/4)}>t'/2-u\right) \,{\rm d}u \right]\,
\int_{\frac{r_0}{2} < |z| \leq r_0 \atop \frac{3r_0}{8} < |x-z| \leq \frac{3r_0}{4} } p^{B\big(x,\frac{3r_0}{2}\big)}(t'/2,x,z) \,{\rm d}z \\
& \geq  {C_5}  \left[\int_0^{t'/2} \pr_0\left(\tau_{B(r_0/4)}>u\right) \, \pr_0\left(\tau_{B(r_0/4)}>t'/2-u\right) \,{\rm d}u \right]
\left|B\left(\frac{r_0}{8}\right)\right| \inf_{\frac{3r_0}{8} < |z| \leq \frac{3r_0}{4}}p^{B\big(\frac{3r_0}{2}\big)}(t'/2,0,z).
\end{align*}
{Here we have used the fact that the set $\left(\overline B(r_0)\setminus \overline B(r_0/2)\right)\cap\left(\overline B(x,3r_0/4)\setminus \overline B(x, 3r_0/8)\right)$ always contains a ball of radius $r_0/8$ (in the last line we first restricted the integration to this ball and then we took the infimum).}
Observe that the last infimum can be estimated exactly in the same way as in \eqref{eq:G_2}. We thus have
\begin{align*}
p^{B\big(\frac{3r_0}{2}\big)}(t'/2,0,z) &
\geq \ex_0\left[X_{\tau_{B(r_0/8)}}\in B(z,r_0/8), \tau_{B(r_0/4)}< t'/2;
p^{B\big(\frac{3r_0}{2}\big)}(t'/2-\tau_{B(r_0/8)}, X_{\tau_{B(r_0/8)}}, z)\right] \\
& \geq{C_5} \int_0^{t'/2} \pr_0\left(\tau_{B(r_0/8)}>u\right) \, \pr_0\left(\tau_{B(r_0/8)}>t'/2-u\right) \,{\rm d}u,
\end{align*}
as long as $\frac{3r_0}{8} < |z| \leq \frac{3r_0}{4}$. In consequence, for $0<s<t$ we have
\begin{eqnarray*}
\inf_{|x| \leq r_0}  p^{B(2r_0)}(t-s,0,x) & \geq &
{\Lambda}(1\wedge({\Lambda} |B(0,r_0/8)|)),
\end{eqnarray*}
where we have denoted
${\Lambda}= {\Lambda}(t',r_0)  : =  {C_5} \int_0^{t'/2} \pr_0\left(\tau_{B(r_0/8)}>u\right) \, \pr_0\left(\tau_{B(r_0/8)}>t'/2-u\right) \,{\rm d}u.$

Finally,
$$
p^{U_R}(t,0,y)
\geq \eta(t,r_0)  \inf_{|z|\leq\frac{R}{2}+r_0} \nu(B(z,r_0)), \quad 2r_0 < |y| < \frac{R}{2},
$$
with $\eta(t,r)$ defined by  \eqref{eq:def_eta}. The Proposition follows.
\end{proof}

In the sequel we will  make use of the following symmetrization of the exponent $\psi$. Denote
\begin{align} \label{eq:Lchexpprof}
\Psi(r) = \sup_{|\xi| \leq r} \psi(\xi), \quad r>0.
\end{align}
It follows from a combination of \cite[Remark 4.8]{bib:Sch} and \cite[Section 3]{bib:Pru} that
there exist constants {$C_6, C_7 > 0$}, independent of the process (i.e. of $A$ and $\nu$), such that
\begin{align} \label{eq:PruitH}
{C_6} H\left(\frac{1}{r}\right) \leq \Psi(r) \leq {C_7} H\left(\frac{1}{r}\right), \;\; r>0, \quad \text{where}
\quad H(r) = \frac{\left\|A\right\|}{r^2}+ \int_{\R^d\setminus \{0\}} \left(1 \wedge \frac{|y|^2}{r^2}\right) \nu({\rm d}y)
\end{align}
($\|A\|$ denotes the {operator} norm of a square matrix $A$).
A direct proof of this estimate with explicit constants can be found in \cite[Lemma 6]{bib:Grz}.
It immediately follows from the definition that $H$ is non-increasing and that the doubling property $H(r) \leq 4 H(2r)$, $r>0$, holds.
In particular, $\Psi(2r) \leq 4 C_6^{-1} C_7 \Psi(r)$, for all $r>0$. Also, by \eqref{eq:PruitH} we get that $\nu \big(B(0,r)^c\big) \leq {C_6}^{-1} \Psi(1/r)$, $r >0$.

\smallskip

\subsection{Processes with polynomially decaying L\'evy measures}
\label{subsec:abs_polyn}

In this subsection we show how our general results translate to the case when the L\'evy measure is
polynomially decaying at infinity. We now give versions of Theorems \ref{thm:upper} and \ref{thm:lower_bound}
specialized to this case. Recall that for a symmetric $\alpha-$stable process with L\'evy-Khinchine exponent $\psi^{(\alpha)},$ by $\lambda_1^{(\alpha)}(U)$ we denote the principal Dirichlet eigenvalue of a set $U,$ and by $\lambda_{(\alpha)}$ -- the infimum of $\lambda_1^{(\alpha)}$ over all open sets of unit measure.

We first consider the class of L\'evy {processes that are close to non-Gaussian symmetric stable processes in the sense of the condition {\bf (C)}}. As we will see later (Lemma \ref{lem:mon}), when {the L\'evy measure of such process} has a density comparable with a nonincreasing function, then its decay at infinity is necessarily stable-like.

\begin{theorem} \label{thm:polynomial_irr}
Let $X$ be a symmetric L\'evy process with characteristic exponent $\psi$ as in \eqref{eq:Lchexp}, with Gaussian coefficient $A$ and L\'evy measure $\nu$ such that
{\bf (C)} with $\alpha\in(0,2)$ and \eqref{eq:psi_inf} hold, i.e.
\begin{itemize}
\item[(i)] $\psi(\xi)=\psi^{(\alpha)}(\xi) + o(|\xi|^\alpha)$ as $|\xi|\to 0$ for some $\alpha \in (0,2)$,
where $\psi^{(\alpha)}(\xi)$ is defined in \eqref{eq:psi-stable},
\end{itemize}
{and}
\begin{itemize}
\item[(ii)] $\lim_{|\xi| \to \infty} \frac{\psi(\xi)}{(\log |\xi|)^2} = \infty$.
\end{itemize}
Further, let $V^{\omega}$ be a Poissonian potential with bounded, compactly supported, nonnegative and nonidentically zero profile $W$.
Then the following hold.
\begin{itemize}
\item[(a)] For any fixed $x \in \R^d$ one has
$$
\limsup_{t\to\infty}\frac{\log u^\omega(t,x)}{t^{\frac{d}{d+\alpha}}}
\leq - \alpha \left(\frac{2}{d+2\alpha}\right)^{\frac{d}{d+\alpha}}
\left(\frac{\rho}{d}\right)^{\frac{\alpha}{d+\alpha}}
\left(\lambda_{(\alpha)}\right)^{\frac{d}{d+\alpha}},
\quad \qpr-\mbox{a.s.}
$$

\item[(b)] If there exist $C_8, C_9, r_0 > 0$ such that for every Borel set $E \subset \R^d$
$$
\nu \big(E \cap \left\{x:0<|x|\leq r_0\right\}\big) \geq {C_8} |E \cap \left\{x:0<|x|\leq r_0\right\}|,
$$
 and
$$
\nu\big(B(x,r_0)\big) \geq {C_9} |x|^{-d-\alpha}, \quad \text{for \ $|x| \geq r_0$,}
$$
then for any fixed $x \in \R^d$ one has
$$
\liminf_{t\to\infty}\frac{\log u^\omega(t,x)}{t^{\frac{d}{d+\alpha}}}\geq
- \left(2\alpha + \frac{9d}{2}\right)\left(\frac{2}{d+2\alpha}\right)^{\frac{d}{d+\alpha}}
\left(\frac{\rho\omega_d}{d}\right)^{\frac{\alpha}{d+\alpha}}\left(\lambda_1^{(\alpha)}(B(0,1))
\right)^{\frac{d}{d+\alpha}}, \quad \qpr-\text{a.s.}
$$
\end{itemize}
\end{theorem}	

\begin{proof} (a)
To prove the upper bound we will apply our general Theorem
\ref{thm:upper}. First, we verify condition \textbf{(U)}. From \cite[Section 3 and (3.2)]{bib:Pru} (see also \cite[Lemma 4.1]{bib:Sch}) and \eqref{eq:PruitH} combined with the basic asymptotic assumption (i) we get
$$
\pr_0(|X_t| > r) \leq \pr_0(\sup_{s \in (0,t]} |X_s| > r) \leq c_1 \, t \,  H(r) \leq c_2 \, t \, \Psi\left(\frac{1}{r}\right) \leq c_3 \frac{t}{r^{\alpha}},
$$
for every $t>0$ and sufficiently large $r >0$, with some contants $c_1,c_2,c_3 > 0$.
Thus \textbf{(U)} holds with $F(r) = r^{-\alpha}$ and $\gamma = 1$.

From Theorem \ref{th:Okura} we also see that
(\ref{eq:IDS_upper}) is satisfied with $\kappa_0=\rho(\lambda_{(\alpha)})^{d/\alpha}.$ To manage the correction terms appearing in the statement
of Theorem \ref{thm:upper}, observe that for given $\kappa >0$ (cf. \eqref{eq:basicf}, \eqref{eq:ha_te})
$$
f_{F,\alpha,\kappa}(r)= \left(\alpha+\frac{d}{2}\right)\left(\frac{d}{\kappa}\right)^{\alpha/d} \, (\log r)^{\frac{\alpha+d}{d}}
\quad \text{for sufficiently large \ $r$},
$$
so that
$$
h_{F,\alpha,\kappa}(t)=\exp\left\{
\left(\frac{\kappa}{d}\right)^{\frac{\alpha}{d+\alpha}}\left(\frac{2}{2\alpha+d}\right)
^{\frac{d}{d+\alpha}}t^{\frac{d}{d+\alpha}}\right\}, \quad \text{for large \ $t$},
$$
and
\[g(t)=g_{F,\alpha,\kappa}(t)= \left(\frac{2}{2\alpha+d}
\left(\frac{\kappa}{d}\right)^{\frac{\alpha}{d}}
\right)^{-\frac{\alpha}{d+\alpha}}\,t^{\frac{d}{d+\alpha}}.\]
By inspection we check that
\begin{eqnarray}\label{eq:g-h0}
 \frac{\log h_{F,\alpha,\kappa}(t)}{g_{F,\alpha,\kappa}(t)} = \left(\frac{\kappa}{d}\right)^{\frac{\alpha}{d}}
\left(\frac{2 }{2\alpha+d}\right).
\end{eqnarray}
Using this observation with $\kappa=\kappa_0,$ from (\ref{eq:IDS_upper}) and (\ref{eq:g-h0})   we get that
 $\qpr$-almost surely
\[\limsup_{t\to\infty}\frac{\log u^\omega(t,x)}{g(t)}\leq
-\frac{2\alpha}{2\alpha+d}  \left(\frac{\kappa_0}{d}\right)^{\frac{\alpha}{d}}
=-\frac{2\alpha}{2\alpha+d}
\left(\frac{\rho}{d}\right)^{\frac{\alpha}{d}}\,\lambda_{(\alpha)},\]
or
\[\limsup_{t\to\infty}\frac{\log u^\omega(t,x)}{t^{\frac{d}{d+\alpha}}}
\leq - \alpha \left(\frac{2}{2\alpha+d}\right)^{\frac{d}{d+\alpha}}
\left(\frac{\kappa_0}{d}\right)^{\frac{\alpha}{d+\alpha}}=
- \alpha\left(\frac{2}{2\alpha+d}\right)^{\frac{d}{d+\alpha}}
\left(\frac{\rho}{d}\right)^{\frac{\alpha}{d+\alpha}}
\left(\lambda_{(\alpha)}\right)^{\frac{d}{d+\alpha}}.
\]
\smallskip

\noindent(b)
Proposition \ref{prop:conditionA} asserts that the assumptions
of Theorem \ref{thm:lower_bound} are satisfied with any $K>1$. To match the asymptotic profile from the upper bound, we take again $F(r)=r^{-\alpha}$, and $R_0=2r_0.$ We shall first obtain the lower bound with any given $\kappa>0,$ and at the end we will choose a suitable $\kappa.$ We first check the assumptions of Corollary \ref{cor:lower}.
{Due to Proposition \ref{prop:dirichlet-by-nu} and the lower bound on $\nu$ in (b)
we see that $G(2r_0, R)\geq c_4 R^{-d-\alpha}$ for large $R$, with $c_4 >0$}, therefore
\begin{equation}
\limsup_{r \to \infty} \frac{\left|\log {G\left(2r_0, \frac{2\sqrt d r}{(\log r)^{\frac{2}{d}+2}}\right)}\right|}{r\wedge|\log F(r)| + (d/2) \log r} \leq \frac{d+\alpha}{d/2 + \alpha},
\end{equation}
and we can take {$Q_1=\alpha, $ $Q_2=\frac{\alpha+d}{\alpha+d/2}$} in \eqref{eq:statement-upper:A}.
This gives that, for any fixed $\kappa>0$
and $x \in \R^d$, $\qpr$-almost surely,
$$
\liminf_{t \to \infty} \frac{\log u^{\omega}(t,x)}{g(t)} \geq -K\left(\frac{\omega_d\rho}{d}\right)^{\alpha/d}\lambda_1^{(\alpha)}(B(0,1)) -\left(\frac{\kappa}{d}\right)^{\alpha/d}\left(\frac{ 2\alpha+8d }{2 \alpha+d}\right),
$$
i.e.
\[\liminf_{t\to\infty}\frac{\log u^\omega(t,x)}{t^{\frac{d}{d+\alpha}}}\geq
-\frac{ {K} \lambda_1^{(\alpha)}(B(0,1))
\left(\frac{\omega_d\rho}{d}\right)^{\frac{\alpha}{d}} + (\alpha+4d)\left(\frac{\kappa}{d}\right)^{\frac{\alpha}{d}}\,
\frac{2}{2\alpha+d}
}{\left(\left(\frac{\kappa}{d}\right)^{\frac{\alpha}{d}}\,
\frac{2}{2\alpha+d}\right)^{\frac{\alpha}{d+\alpha}}}.
\]
Letting $K\downarrow 1$ through rationals, we  get this statement with $K=1.$
To match the upper bound, we take $\kappa= \kappa_0=\rho\omega_d\left(\lambda_1^{(\alpha)}(B(0,1))\right)^{\frac{d}{\alpha}}$
i.e. the value corresponding to that in the upper bound.
We  easily check that with this choice of $\kappa$ we get
$$
\liminf_{t\to\infty}\frac{\log u^\omega(t,x)}{t^{\frac{d}{d+\alpha}}}\geq
- \left(2\alpha+\frac{9d}{2}\right) \left(\frac{2}{2\alpha+d}\right)^{\frac{d}{d+\alpha}}
\left(\frac{\rho\omega_d}{d}\right)^{\frac{\alpha}{d+\alpha}}\left(\lambda_1^{(\alpha)}(B(0,1))
\right)^{\frac{d}{d+\alpha}}, \quad \qpr-\text{a.s.}
$$
The proof is complete.
\end{proof}

\begin{remark}{\rm

\noindent
\begin{itemize}
\item[(1)]
When $\inf_{|\xi|=1} \psi_A(\xi)> 0$, then (ii) automatically holds and needs not to be assumed a priori.
\item[(2)]
If $\nu({\rm d}x)=\nu(x){\rm d}x$ and there exists a nonincreasing profile function $g:(0,\infty) \to (0,\infty)$ such that $\nu(x) \asymp g(|x|)$, $x \in \R^d \setminus \left\{0\right\}$, then  condition (i) implies that {
\begin{align} \label{eq:monot_psi}
 \quad g(r) \asymp \Psi(1/r) r^{-d} \asymp r^{-d-\alpha}, \quad \text{for all \ $r \geq 1$, }
\end{align}
where $\Psi$ is the symmetrization of  $\psi$ defined in \eqref{eq:Lchexpprof}.
In particular, the assumption in part (b) of the theorem automatically holds and can be omitted.
A short proof of \eqref{eq:monot_psi} is given in Lemma \ref{lem:mon} below.}
\item[(3)]
When the process $X^{(\alpha)}$ is isotropic, then our upper and lower bound differ by just a multiplicative constant
(see the comment following Theorem \ref{th:Okura}).
However, we were not able to get the almost sure convergence in this theorem, even for isotropic processes.
We do not know whether it is a flaw of the method, or if it is an intrinsic feature of the functional, signaling the
existence of some `small deviations' phenomenon here. {The detailed study of this case is the purpose of an ongoing project.}
\item[(4)]
In the final part of the proof above we could do better:  choose $\kappa=a\kappa_0$ and then optimize over $a>0.$ This  gives:
$$
\liminf_{t\to\infty}\frac{\log u^\omega(t,x)}{t^{\frac{d}{d+\alpha}}}
\geq
- (\alpha+4d)^{\frac{\alpha}{d+\alpha}}\left(
\left(\frac{d}{\alpha}\right)^{\frac{\alpha}{d+\alpha}}+\left(\frac{\alpha}{d}
\right)
^{\frac{d}{d+\alpha}}\right)
\left(\frac{\rho\omega_d}{d}\right)^{\frac{\alpha}{d+\alpha}}\left(\lambda_1^{(\alpha)}(B(0,1))
\right)^{\frac{d}{d+\alpha}},  \quad \qpr-\mbox{a.s.}
$$
This is the best lower bound that can be derived from Theorem \ref{thm:lower_bound}.
\end{itemize}
}
\end{remark}

\begin{lemma} {\label{lem:mon}
Let $X$ be a symmetric L\'evy process with characteristic exponent $\psi$ as in \eqref{eq:Lchexp}, with Gaussian coefficient $A \equiv 0$ and L\'evy measure $\nu({\rm d}x)=\nu(x){\rm d}x,$ and such that there exists a nonincreasing profile function $g:(0,\infty) \to (0,\infty)$ for which $\nu(x) \asymp g(|x|)$, $x \in \R^d \setminus \left\{0\right\}$. Then condition (i) of Theorem \ref{thm:polynomial_irr} implies \eqref{eq:monot_psi}.}
\end{lemma}

\begin{proof}
First note that by \cite[Lemma 5 (a)]{bib:KS14} and (i) of Theorem \ref{thm:polynomial_irr}, we have $\Psi(|x|) \asymp \psi(x) \asymp |x|^{\alpha}$, whenever $|x| \leq 1$. As in this case for $r\geq 1$ we have (see \eqref{eq:PruitH})
\[\frac{1}{r^\alpha}\asymp\Psi\left(\frac{1}{r}\right)\asymp H(r)=\int_{\mathbb R^d\setminus\{0\}}\left(1\wedge\frac{|y|^2}{r^2}\right)\nu({\rm d}y),\]
so that there are constants $c_1, c_2 >0$ for which
\begin{align} \label{eq:lem_stable}
\frac{c_1}{r^{\alpha}} \leq \frac{1}{r^2} \int_{|y| < r} |y|^2 g(|y|) dy + \int_{|y| \geq r} g(|y|) dy \leq \frac{c_2}{r^{\alpha}}, \quad r \geq 1.
\end{align}
By the monotonicity of $g$, for $r\geq 1$ we get
\[\frac{c_2}{r^\alpha}\geq
\frac{1}{r^2}\int_{|y|<r} |y|^2 g(|y|){\rm d}y \geq {c_3}g(r)r^d.\]
To show the opposite inequality, observe that for any $u \in (0,1)$  and  $s > 1$  from the
 monotonicity of $g$ and  the upper bound in \eqref{eq:lem_stable}, we may write
$$
\int_{|y| \geq r} g(|y|) dy = \int_{r \leq |y| < sr} g(|y|) dy + \int_{|y| \geq sr} g(|y|) dy \leq c_4 (s^d-1) g(r) r^d + \frac{c_2}{s^{\alpha}} \frac{1}{r^{\alpha}}, \quad r \geq 1,
$$
for some $c_4 >0$. Similarly,
\begin{align*}
\frac{1}{r^2} \int_{|y| < r} |y|^2 g(|y|) dy & = u^2 \frac{1}{(ur)^2} \int_{|y| < ur} |y|^2 g(|y|) dy + \frac{1}{r^2} \int_{ur \leq |y| < r} |y|^2 g(|y|) dy \\ & \leq \frac{c_2 u^{2-\alpha}}{r^{\alpha}} + c_4 (1-u^d) g(ur) r^d, \qquad r \geq \frac{1}{u}.
\end{align*}
Adding these two estimates, from \eqref{eq:lem_stable} we obtain,
using also the monotonicity of $g:$
 \[\frac{c_1}{r^\alpha}\leq \frac{c_2u^{2-\alpha}}{r^\alpha} +\frac{c_2}{s^\alpha r^\alpha} + c_4 (s^d-u^d)g(ur) r^d.\]
Choosing $u_0\in(0,1)$  so small and $s_0>1$ so large  that $c_2 ({u_0}^{2-\alpha} + {s_0}^{-\alpha}) \leq c_1/2$, we finally obtain
$$
\frac{c_1}{2 r^{\alpha}} \leq   c_4 (s_0^d-u_0^d) g(u_0r) r^d, \quad r \geq \frac{1}{u_0},
$$
and, as a consequence, $r^{-d-\alpha} \leq \frac{2c_4(s_0^d - u_0^d)}{c_1 u^{d-\alpha}} g(r)$, for every $r \geq 1$, which is the claimed inequality. This completes the proof.
\end{proof}

We now illustrate our Theorem \ref{thm:polynomial_irr} with several examples.

\begin{example} {\rm {\textbf{(Absolutely continuous perturbations of isotropic stable processes)}} \label{ex:polynomial_irr}
 {
\noindent
\begin{itemize}
\item[(1)]\emph{Non-Gaussian isotropic stable processes.} Let $\psi(\xi) = |\xi|^{\delta}$, with $\delta \in (0,2)$. In this case $\nu({\rm d}x) = C_{d,\delta} |x|^{-d-\delta} {\rm d}x$ and we have to take $\alpha = \delta$ and $\psi^{(\alpha)} = \psi$. In particular, the assumptions of  Theorem \ref{thm:polynomial_irr} hold.
\item[(2)] \emph{Mixture of isotropic stable processes (possibly with Brownian component).} Let $\psi(\xi) = a_0 |\xi|^2 + \sum_{i=1}^{n} a_i |\xi|^{\alpha_i}$, $n \in \N$, $a_0 \geq 0$, and $a_i >0$, $\alpha_i \in (0,2)$, for $i =1,...,n$. Then the assumptions of Theorem \ref{thm:polynomial_irr} are satisfied with $\alpha =  \alpha_{i_{min}}:=\min_i \alpha_i$ and $\psi^{(\alpha)}(\xi) = a_{i_{min}} |\xi|^{\alpha_{i_{min}}}$.
\item[(3)] \emph{Isotropic geometric stable process with Gaussian component.} Let $\psi(\xi)= \xi \cdot A \xi + \log(1+|\xi|^{\delta})$, with $A$ such that $\inf_{|\xi|=1} \xi \cdot A \xi >0$ and $\delta \in (0,2)$. Again, Theorem \ref{thm:polynomial_irr} applies with $\alpha = \delta$ and $\psi^{(\alpha)}(\xi) = |\xi|^{\alpha}$.
\end{itemize}
}}
\end{example}

\begin{example} {\rm {\textbf{(More general perturbations of symmetric stable processes)}} \label{ex:polynomial_irr_2}

\noindent
{Let $\delta \in (0,2)$ and let $n$ be a symmetric finite measure on the unit sphere $S^{d-1}$ such that
$$
n(B(\theta,r) \cap S^{d-1}) \geq c_0 r^{d-1}, \quad \theta \in S^{d-1}, \ \ r \in (0,1/2],
$$
for some constant $c_0>0$.
Denote the corresponding stable L\'evy measure by $\nu^{(\delta)}({\rm d}r {\rm d}\theta) = n({\rm d} \theta)r^{-1-\delta} {\rm d}r$.
Note that we do not impose similar growth condition on $n$ from above, which means that $\nu^{(\delta)}$ is not necessarily absolutely continuous with respect to the Lebesgue measure. Furthermore, let $\nu_{\infty}$ be a (non-necessarily infinite) measure on $\R^d \setminus \left\{0\right\}$ such that
\begin{align} \label{eq:dom_by_delta}
\int_{\R^d \setminus \left\{0\right\}}(1 \wedge (r|z|)^2)\nu_{\infty}({\rm d}z)= o(r^\delta) \quad \text{as \ $r \to 0$,}
\end{align}
and consider a symmetric L\'evy process with L\'evy-Khinchine exponent $\psi$ as in \eqref{eq:Lchexp} with arbitrary diffusion matrix $A$ and L\'evy measure $\nu = \nu^{(\delta)} + \nu_{\infty}$.
Then the assumptions of Theorem \ref{thm:polynomial_irr} hold with $\alpha = \delta$ and
$\psi^{(\alpha)}(\xi) = \int_0^\infty\int_{S^{d-1}} (1-\cos(\xi\cdot r\theta))\,n({\rm d}\theta)\frac{{\rm d}r}{r^{\alpha+1}}$.
Typical examples of measures $\nu_{\infty}$ satisfying \eqref{eq:dom_by_delta} are as follows.
\begin{itemize}
\item[(1)] \emph{{Other} stable L\'evy measures}: $\nu_{\infty} := \nu^{(\widetilde \delta)}({\rm d}r {\rm d}\theta) = \widetilde n({\rm d} \theta)r^{-1- \widetilde \delta} {\rm d}r$, where $\widetilde \delta \in (\delta, 2)$ and $\widetilde n$ is a symmetric finite measure on $S^{d-1}$.
\item[(2)] \emph{Product L\'evy measures with profiles with sufficiently fast decay at infinity}:  $\nu_{\infty} := \widetilde n({\rm d} \theta) f(r) {\rm d}r$, where $\widetilde n$ is a symmetric finite measure on $S^{d-1}$ and $f:(0,\infty) \to [0,\infty)$ is a function such that\linebreak $\int_0^{\infty} s^2 f(s) {\rm d}s < \infty$.
\item[(3)] \emph{Purely discrete L\'evy measures}. Let $\left\{v_k: k=1,..,k_0 \right\}$ be a family of $k_0 \in \N$ vectors in $\R^d$. For $q>0$ denote
$$
A_q=\left\{x \in \R^d: x=2^{q n} v_k, \ \text{where} \ n \in \Z, \ k = 1, ..., k_0 \right\}
$$
and
$$
f(s):= \1_{[0,1]}(s) \cdot s^{-\theta/q}  +  \1_{(1,\infty)}(s) \cdot s^{-\beta/q} , \quad s > 0,
$$
with $\theta \in (0,2q)$ and $\beta > 2q$. Then one can take
$$
\nu_{\infty}({\rm d}y):= \int_{\R^d} f(|y|) \delta_{A_q}({\rm d}y) = \sum_{y \in A_q} f(|y|).
$$
\end{itemize}
}
}
\end{example}

We now turn to the class of processes with L\'evy measures that have second moment finite. In this case, for a given Gaussian matrix $A=(a_{ij})_{1 \leq i, j \leq d},$ the coefficients
\begin{align} \label{eq:lambda_for_A_nu}
\widetilde A = (\widetilde a_{ij})_{1 \leq i, j \leq d}, \quad \text{with \ $\widetilde a_{ij} = a_{ij} + \frac{1}{2} \int_{\R^d} y_i y_j \nu(y)dy$},
\end{align}
are well defined (see Proposition \ref{prop:suff_for_Okura} (i)).
In what follows, by $\lambda_1^{(2)}(U)$ we will always denote the principal eigenvalue of the diffusion process with characteristic exponent $\psi^{(2)}(\xi) = \xi \cdot \widetilde A \xi$, killed on leaving an open bounded set $U \subset \R^d$.
Also, $\lambda_{(2)}$ denotes the infimum of $\lambda_1^{(2)}(U)$ over all open sets $U$ of unit measure {(recall the comment after Theorem \ref{th:Okura})}.

We now discuss the case of L\'evy measures with polynomial tails whose decay at infinity
is faster than stable.
To avoid some technical difficulties and for more clarity, in the theorem below we restrict our attention to the absolutely continuous case {(cf. Example \ref{ex:ex_irr} (2))}.

\begin{theorem} \label{thm:polynomial}
Let $X$ be a symmetric L\'evy process with characteristic exponent $\psi$ as in \eqref{eq:Lchexp}, with defining parameters $A = (a_{ij})_{1 \leq i, j \leq d}$ and $\nu({\rm d}x)=\nu(x){\rm d}x.$ {Assume that}
\begin{itemize}
\item[(i)] either {$\inf_{|\xi|=1} \psi_A(\xi) > 0$, $\liminf_{|\xi| \to \infty} \frac{\psi_{\nu}(\xi)}{\log |\xi|} > 0$ or $\psi_A(\xi) \equiv 0$, $\lim_{|\xi| \to \infty} \frac{\psi_{\nu}(\xi)}{(\log |\xi|)^2} = \infty$},
\item[(ii)] there exist {$C_{10} \geq C_{11} > 0 $} and $\delta_1 \geq \delta_2 >2$ such that
$$
 {{C_{11}}}\left(\frac{1}{1+|x|^{d+\delta_1}}\right) \leq \nu(x) \leq \frac{{C_{10}}}{|x|^{d+\delta_2}}, \quad x \in \R^d.
$$
\end{itemize}
Moreover, let $V^{\omega}$ be a Poissonian potential with bounded, compactly supported, nonnegative and nonidentically zero profile $W$.
Then, for any fixed $x \in \R^d$, one has
$$
\limsup_{t\to\infty}\frac{\log u^\omega(t,x)}{t^{\frac{d}{d+2}}}
\leq - \delta_2\left(\frac{2}{2\delta_2+d}\right)^{\frac{d}{d+2}}
\left(\frac{\rho}{d}\right)^{\frac{2}{d+2}}
\left(\lambda_{(2)}\right)^{\frac{d}{d+2}},
\quad \qpr-\mbox{a.s.}
$$
and
$$
\liminf_{t\to\infty}\frac{\log u^\omega(t,x)}{t^{\frac{d}{d+2}}}\geq
- \left(2\delta_1+\frac{9d}{2}\right)\left(\frac{2}{2\delta_1+d}\right)^{\frac{d}{d+2}}
\left(\frac{\rho\omega_d}{d}\right)^{\frac{2}{d+2}}\left(\lambda_1^{(2)}(B(0,1))
\right)^{\frac{d}{d+2}}, \quad \qpr-\text{a.s.}
$$
\end{theorem}

\begin{proof}
By Proposition \ref{prop:suff_for_Okura} (i) not only the coefficients
 $a_{ij}$ given by \eqref{eq:lambda_for_A_nu} are finite, but also  one has
$$
\psi(\xi) = \xi \cdot \widetilde A \xi + o(|\xi|^2),
$$
i.e. the basic asymptotic assumption \textbf{(C)} holds true with $\alpha = 2$. {Also, \eqref{eq:psi_inf} is satisfied in both cases of (i).}

As usual, to establish the upper bound we apply our general Theorem
\ref{thm:upper}. When $A \equiv 0$, then from Lemma \ref{lem:density_estimate} below we get $p(t,x)\leq c_1 t^{\delta_2/2}|x|^{-d-\delta_2}$, $x \in \R^d \setminus \left\{0\right\}$, $t\geq t_1,$ for some $t_1>0$, so that \textbf{(U)} holds with {$F(r) = r^{-\delta_2}$} and $\gamma = \delta_2/2$. When $\inf_{|\xi|=1} \xi \cdot A \xi > 0$, then the same is true by Proposition \ref{prop:gauss_est}. Also, by Theorem \ref{th:Okura}, (\ref{eq:IDS_upper}) holds true with $\kappa_0=\rho(\lambda_{(2)})^{d/2}$ and the proof of the upper bound can completed by following the argument in the proof of the upper bound in Theorem \ref{thm:polynomial_irr} above, with the function
(for large $r$)
$$
f_{F,2,\kappa}(r)= \left(\delta_2+\frac{d}{2}\right)\left(\frac{d}{\kappa}\right)^{2/d} \, (\log r)^{\frac{2+d}{d}}.
$$

To get the lower bound, it is enough to observe that by {Proposition \ref{prop:dirichlet-by-nu}} and the  bound on the density $\nu(x)$
we have {$G(2,R)\geq c_2 R^{-d-\delta_1}$ }for large $R$. Indeed, the rest of the proof follows the lines of the second part of the justification of the lower bound in Theorem \ref{thm:polynomial_irr} with  profile function {$F(r) = r^{-\delta_1}$}.
\end{proof}

To complete the proof of the above theorem we need to prove the following lemma.

\begin{lemma}\label{lem:density_estimate}
Let $X$ be a symmetric L\'evy process with characteristic exponent $\psi$ as in \eqref{eq:Lchexp},
with defining parameters $A \equiv 0$ and $\nu({\rm d}x)=\nu(x){\rm d}x$, such that
{$\liminf_{|\xi| \to \infty} \frac{\psi(\xi)}{\log |\xi|} > 0$}. If there exist {$C_{12}-C_{14} >0$} and
$\delta > 2$ such that $\nu(x) \leq {C_{12}} |x|^{-d-\delta}$ for $x \in \R^d \setminus \left\{0\right\}$
{and $\nu(x) \geq C_{13}$ \, for $|x| \leq C_{14}$}, then
then there exist {$C_{15}, t_1>0$}
such that
\begin{align*}\label{eq:est-dens-1}
 p(t,x)\leq {C_{15}} \frac{t^{\frac{\delta}{2}}}{|x|^{d+\delta}}, \quad x \in \R^d \setminus \left\{0\right\}, \ t>t_1.
\end{align*}
\end{lemma}

\begin{proof}
The assumption {$\liminf_{|\xi| \to \infty} \frac{\psi(\xi)}{\log |\xi|} > 0$ immediately gives that $e^{-t\psi(\cdot)} \in L^1(\R^d)$, for sufficiently large} $t>0$, and also that $\nu(\R^d \backslash \left\{0\right\}) = \infty$. Thus, by the Fourier inversion formula, $p(t,x)$ exist and are bounded for all $t$ {large enough}. To find an upper bound on $p(t,x)$, we use \cite[Theorem 1]{bib:KS13}. Its assumption (1) follows directly from the upper bound on the density $\nu(x)$ with the profile function $f(r) = r^{-d-\delta}$ and (2) can be directly derived from the monotonicity and the doubling property of such $f$. Indeed,  for every $s, r >0$ we may write
$$
\int_{|y|>r} f\left(s \vee |y| - \frac{|y|}{2}\right) \nu(y)dy
= \left(\int_{|y|>r \atop |y| \leq s} + \int_{|y|>r \atop |y| > s}\right) f\left(s \vee |y| - \frac{|y|}{2}\right) \nu(y)dy =: I_1 + I_2.
$$
Since $f$ is nonincreasing, both integrals $I_1$ and $I_2$ can be easily estimated by $f(s/2) \int_{|y|>r} \nu(y)dy$, which is smaller or equal to $c_1 f(s) \Psi(1/r)$, for all $s, r >0$. Thus the assumption (2) holds true.

It remains to justify the last assumption (3). First note that by the upper estimate of the density $\nu(x)$ and Proposition \ref{prop:suff_for_Okura} (i) one has
$$
\psi(\xi) = \xi \cdot \widetilde A \xi + o(|\xi|^2), \quad \text{with \ $\widetilde A = (\widetilde a_{ij})_{1 \leq i,j\leq d}$,
\ where \ $a_{ij}= \frac{1}{2} \int_{\R^d} y_i y_j \nu(y)dy$}.
$$
Since $\nu(x)$ is {separated from zero} around the origin, this together with {$\liminf_{|\xi| \to \infty} \frac{\psi(\xi)}{\log |\xi|} > 0$,} imply that there exist $0 < r_1 < 1 < r_2$ such that
$$
c_2 |\xi|^2\leq \psi(\xi) \leq c_3 |\xi|^2 \quad \text{for} \quad |\xi| \leq r_1 \quad \text{and} \quad \psi(\xi) \geq c_4 \log |\xi| \quad \text{for} \quad |\xi| \geq r_2,
$$
with some constants $c_2,..., c_4  >0$. The above bounds and the fact that $\inf_{r_1 \leq |\xi| \leq r_2} \psi(\xi) > 0$ immediately give that  for every $t > 0$ we may write
\begin{align*}
\int_{\R^d} {\rm e}^{-t \psi(\xi)} |\xi| d \xi
\leq \int_{|\xi| < r_1} {\rm e}^{-c_2 t |\xi|^2} |\xi| d \xi
 + {\rm e}^{-t \inf_{r_1\leq|\xi| \leq r_2} {\psi(\xi)}} \int_{r_1 \leq |\xi| \leq r_2} |\xi| d \xi +\int_{|\xi| > r_2} e^{-c_4 t \log |\xi|} |\xi| d \xi.
\end{align*}
When $t>\frac{d+1}{c_4},$ then the last integral is convergent, and, moreover, we get that there exists $c_5>0$ such that the estimate
\[
\int_{\R^d} {\rm e}^{-t \psi(\xi)} |\xi| d \xi\leq c_5 t^{-\frac{d+1}{2}}\] holds for large $t>0$.
Since $\Psi^{-1}(1/t) \asymp t^{-1/2}$ for sufficiently large $t$, this in fact gives that there is $t_1 > 0$
such that the assumption (3) in \cite[Theorem 1]{bib:KS13} holds with $T=[t_1,\infty)$. Consequently we have, with $h(t)\asymp t^{1/2}$ and $\gamma =d,$
(due to the symmetry of the process the correction term $tb_{h(t)}$ is not present):
$$
p(t,x) \leq c_6 \left(\frac{t}{|x|^{d+\delta}} + t^{-d/2} e^{-c_7 \frac{|x|}{t^{1/2}} \log\left(1+ c_7 \frac{|x|}{t^{1/2}}\right)}\right),
\quad x \in \R^d \setminus \left\{0\right\}, \ t \geq t_1.
$$
We estimate the exponentially-logarithmic term above. If $|x| \geq t^{1/2}$, then
$$
t^{-d/2} e^{-c_{7}\frac{|x|}{t^{1/2}} \log \left(1+c_{7}\frac{|x|}{t^{1/2}}\right)}
\leq c_{8} t^{-d/2} (t^{1/2}/|x|)^{d+\delta} = c_{8} t^{\delta/2}/|x|^{d+\delta},
$$
for a constant $c_{8} >0$. On the other hand, for $|x| \leq t^{1/2}$ we simply have
$$
t^{-d/2} e^{-c_{7}\frac{|x|}{t^{1/2}} \log \left(1+c_{7}\frac{|x|}{t^{1/2}}\right)}
\leq t^{-d/2} \leq t^{-d/2}(t^{1/2}/|x|)^{d+\delta} =  t^{\delta/2}/|x|^{d+\delta}.
$$
Altogether,
$$
p(t,x) \leq c_{9} \frac{t^{\frac{\delta}{2}}}{|x|^{d+\delta}}, \quad x \in \R^d \setminus \left\{0\right\}, \ t \geq t_1.
$$
The statement follows.
\end{proof}

\begin{remark}{\rm

\noindent

Similarly as before, in the final part of the proof of Theorem \ref{thm:polynomial} above we could get a better bound. Indeed, by choosing $\kappa=a\kappa_0$ and optimizing over $a>0$, one has
$$
\liminf_{t\to\infty}\frac{\log u^\omega(t,x)}{t^{\frac{d}{d+2}}}
\geq
- (\delta_1+4d)^{\frac{2}{d+2}}\left(
\left(\frac{d}{2}\right)^{\frac{2}{d+2}}+\left(\frac{2}{d}
\right)
^{\frac{d}{d+2}}\right)
\left(\frac{\rho\omega_d}{d}\right)^{\frac{2}{d+2}}\left(\lambda_1^{(2)}(B(0,1))
\right)^{\frac{d}{d+2}},  \quad \qpr-\mbox{a.s.}
$$
}
\end{remark}

\begin{example} {\rm {\emph{\textbf{(Layered stable process)}}} \label{ex:polynomial}

\noindent
Our Theorem \ref{thm:polynomial} above can be illustrated by considering a symmetric L\'evy process with L\'evy measure $\nu({\rm d}x)=\nu(x){\rm d}x$ such that $\nu(x) \asymp \1_{\left\{|x| \leq 1 \right\}} |x|^{-d-\eta} +  \1_{\left\{|x| > 1 \right\}}  |x|^{-d-\delta}$, $x \in \R^d \backslash \left\{0\right\}$, with $\eta \in (0,2)$ and $\delta>2$. Such processes are often called \emph{layered stable}.}
\end{example}

All the  results above say that when $\psi(\xi)=\psi^{(\alpha)}(\xi) +o(|\xi|^\alpha),$ $\alpha\in (0,2],$
$\xi\to 0,$ and the L\'{e}vy measure decays polynomially at infinity,
then the quenched rate of  convergence
of $u^\omega(t,x)$ as $t\to\infty$ is the same as its annealed rate
of convergence (cf. (\ref{eq:dv})).

\subsection{Processes with L\'evy measures lighter than polynomial at infinity}
\label{subsec:abs_fast}

\noindent
We now show that when the tail of $\nu$ at infinity is
lighter than polynomial, then the almost sure behaviour qualitatively changes and it is no longer true
that it coincides with the annealed behaviour. In this case we are often able to get the convergence, i.e.
to derive precisely the main term in the asymptotics.

First we discuss the example of a L\'evy measure which decays at infinity faster than polynomially but still slower than {stretched-exponentially.

\begin{theorem} \label{thm:exp_log}
Let $X$ be a symmetric L\'evy process with characteristic exponent $\psi$ as in \eqref{eq:Lchexp} with  Gaussian coefficient $A = (a_{ij})_{1 \leq i, j \leq d}$ such that either  $A \equiv 0$ or $\inf_{|\xi|=1} \xi \cdot A \xi > 0,$ and a symmetric L\'evy measure $\nu({\rm d}x)=\nu(x){\rm d}x$
such that there exist $\theta>0$, $\beta >1$, $\delta \in (0,2)$ satisfying
\begin{align} \label{eq:profile_exp_log_1}
\nu(x) \asymp \1_{\left\{|x| \leq 1 \right\}} |x|^{-d-\delta} +  \1_{\left\{|x| > 1 \right\}} e^{-\theta(\log|x|)^{\beta}}, \quad x \in \R^d \backslash \left\{0\right\}.
\end{align}
Let $V^{\omega}$ be a Poissonian potential with bounded, compactly supported, nonnegative and nonidentically zero profile $W$.
Then, for any fixed $x \in \R^d$, one has
\begin{align*}
\limsup_{t \to \infty} \frac{\log u^{\omega}(t,x)}{t^{\frac{\beta d}{2+d \beta}}} \leq - \, \theta^{\frac{2}{2 + d\beta}}\left(\frac{\rho}{d}\right)^{\frac{2 \beta }{2 + d \beta}} \left(\lambda_{(2)}\right)^{\frac{d \beta }{2 + d \beta}}, \quad \qpr-\text{a.s.},
\end{align*}
and
\begin{align*}
\liminf_{t\to\infty}\frac{\log u^\omega(t,x)}{t^{\frac{\beta d}{2+d \beta}}}\geq
- 2 \theta^{\frac{2}{2 + d\beta}}\left(\frac{\omega_d \rho}{d}\right)^{\frac{2 \beta }{2 + d \beta}} \left(\lambda_1^{(2)}(B(0,1))\right)^{\frac{d \beta }{2 + d \beta}} , \quad \qpr-\text{a.s.},
\end{align*}
where $\lambda_1^{(2)}(B(0,1))$ and $\lambda_{(2)}$ correspond to the diffusion process with Gaussian matrix $\widetilde A$ given by \eqref{eq:lambda_for_A_nu}.
\end{theorem}

\begin{proof}
First of all we observe that indeed by Proposition \ref{prop:suff_for_Okura} (i) one has
$$
\psi(\xi) = \xi \cdot \widetilde A \xi + o(|\xi|^2),
$$
i.e. the basic asymptotic assumption \textbf{(C)} holds true. By \cite[Lemma 5 (a)]{bib:KS14}, we have $\Psi_{\nu}(x) \asymp \psi_{\nu}(|x|)$, $x \in \R^d$, where $\Psi_{\nu}$ is the symmetrization of $\psi_{\nu}$ introduced in in \eqref{eq:Lchexpprof}. From \eqref{eq:PruitH} we thus have
\begin{align} \label{eq:exp_log_eq1}
\psi(x) \geq \psi_{\nu}(x) \geq c_1 \int_{|y| > 1/|x|} \nu(y) dy \asymp {|x|^{\delta}}, \quad \text{for \ $|x|$ \ large enough}.
\end{align}
In particular, \eqref{eq:psi_inf} holds true. We now address the upper bound and the lower bound separately.

\textsc{The upper bound}.
Similarly as before, we use our general Theorem \ref{thm:upper}. First we need to check that \textbf{(U)} holds true.
When $A \equiv 0$, then it follows from Lemma \ref{lem:tran_dens_ub} below
that there exist $c_2 >0$, $c_3 \in (0,1/4]$, $r_0, t_1 >0$ such that
\begin{align} \label{eq:exp_log_eq2}
{p(t,x)}\leq c_2 t e^{- \theta (\log(c_3 |x|))^{\beta}}, \quad |x|>r_0 \vee t, \ t\geq t_1,
\end{align}
which implies \textbf{(U)} with $\gamma=1$ and $F(r)={e^{- \theta (\log(c_3 r))^{\beta}} r^d (\log c_3 r)^{-(\beta-1)}}$ (clearly, such profile $F$ is strictly increasing for $r > \widetilde r_0$ with sufficiently large $\widetilde r_0 \geq r_0$). When $\inf_{|\xi|=1} \xi \cdot A \xi > 0$, then we get from Proposition \ref{prop:gauss_est} that the same is true with the profile function $\widetilde F(r) = e^{-\theta (\log(\widetilde c_3 r))^{\beta}} r^d (\log \widetilde c_3 r)^{-(\beta-1)}$ for some $\widetilde c_3 \in (0,c_3)$ and  same $\gamma$. Thanks to \eqref{eq:exp_log_eq1} and \textbf{(C)}, by Theorem \ref{th:Okura} we also have
$$
\lim_{\lambda\to 0} \lambda^{d/2}\log N^D(\lambda)=-\rho (\lambda_{(2)})^{d/2},
$$
where $\lambda_{(2)}$ is determined by the variational formula \eqref{eq:lambda} with $\lambda_1^{(2)}(G)$ corresponding to the diffusion process with exponent $\psi^{(2)}(\xi) = \xi \cdot \widetilde A \xi$.

We are now in a position to derive the claimed upper bound. Indeed, with the preparation above, by our general Theorem \ref{thm:upper} and Corollary \ref{cor:upper}, we get
\begin{align} \label{eq:exp_log_eq3}
\limsup_{t \to \infty} \frac{\log u^{\omega}(t,x)}{g(t)} \leq -\left(\frac{\rho}{d}\right)^{2/d} \lambda_{(2)}, \quad \qpr-\text{a.s.},
\end{align}
with $g(t)= t/(\log h_{F,\alpha,\kappa_0}(t))^{2/d}$, where $h_{F,\alpha,\kappa_0}$ is the inverse function to $f_{F,\alpha,\kappa_0}$ given by \eqref{eq:basicf} with $\alpha = 2$, $\kappa_0 = \rho (\lambda_{(2)})^{d/2}$ and $F(r) =  e^{- \theta(\log(c_3 r))^{\beta}}  r^d (\log c_3 r)^{-(\beta-1)}$ (as we will see below, here the concrete value of $c_3$ is irrelevant).
By \eqref{eq:basicf}, since $\beta>1,$  we have
$$
f_{F,2,\kappa_0}(r) \approx \left(\theta (\log(c_3r))^{\beta} + \frac{d}{2} \log r \right) \left(\frac{d \log r}{\kappa_0}\right)^{\frac{2}{d}} \approx \theta(\log r)^\beta
\left(\frac{d \log r}{\kappa_0}\right)^{\frac{2}{d}}, \quad \text{for large \ $r$.}
$$
Thus, by direct asymptotic calculations, we obtain that
$$
\log h_{F,\alpha,\kappa_0}(t) \approx  \left(\frac{1}{\theta}\right)^{\frac{d}{2 + d\beta}} \left(\frac{\kappa_0}{d}\right)^{\frac{2}{2 + d\beta}} t^{\frac{d}{2+d\beta}},
$$
and consequently,
\begin{align} \label{eq:exp_log_eq4}
g(t) \approx \left(\frac{1}{{\theta}} \left(\frac{\kappa_0}{d}\right)^{\frac{2}{d}}\right)^{-\frac{2}{2 + d \beta}} t^{\frac{\beta d}{2+d\beta}}.
\end{align}
Since $\kappa_0 = \rho (\lambda_{(2)})^{d/2}$, in light of \eqref{eq:exp_log_eq3}, this gives
$$
\limsup_{t \to \infty} \frac{\log u^{\omega}(t,x)}{t^{\frac{\beta d}{2+d \beta}}} \leq - \, {\theta}^{\frac{2}{2 + d\beta}}\left(\frac{\rho}{d}\right)^{\frac{2 \beta }{2 + d \beta}} \left(\lambda_{(2)}\right)^{\frac{d \beta }{2 + d \beta}}. \quad \qpr-\text{a.s.},
$$

\smallskip

\noindent\textsc{The lower bound}.
First recall that at the beginning of the proof we verified the basic asymptotic assumption \textbf{(C)}. In view of Proposition \ref{prop:conditionA} it gives that the assumptions of our general Theorem \ref{thm:lower_bound} (and Corollary \ref{cor:lower} as well) are satisfied with any $K>1$. To match the asymptotic profile from the upper bound, it is enough to take $F(r) = e^{- \theta(\log r)^{\beta}}$. Similarly as before, we first proceed with an arbitrary fixed $\kappa>0,$ and in the concluding part of the proof we will choose a suitable $\kappa$. Condition \eqref{eq:ass1} of  Proposition \ref{prop:dirichlet-by-nu} is satisfied, so that  we also have that there exists $c_4>0$ such that $G(1,R)\geq c_4 e^{- \theta(\log(R/2))^{\beta}}$ for  sufficiently large $R$.

We now verify the assumptions of Corollary \ref{cor:lower}. First observe that one has $Q_1 = \infty$ in \eqref{eq:A_lower}. Moreover, since
$$
\limsup_{r \to \infty} \frac{\left|\log G\left(1,\frac{2\sqrt d r}{(\log r)^{\frac{2}{d}+2}}\right)\right|}{r\wedge|\log F(r)| + (d/2) \log r} \leq 1,
$$
we also have $Q_2=1$ in \eqref{eq:B_lower}. By \eqref{eq:statement-upper:A}, this yields that for any fixed $\kappa>0$ and $x \in \R^d$
$$
\liminf_{t \to \infty} \frac{\log u^{\omega}(t,x)}{g(t)} \geq -K\left(\frac{\omega_d\rho}{d}\right)^{2/d}\lambda_1^{(2)}(B(0,1)) -\left(\frac{\kappa}{d}\right)^{2/d}, \quad \qpr-\text{a.s.}
$$
Recall that here $\lambda_1^{(2)}(B(0,1))$ corresponds to the diffusion process determined by $\psi^{(2)}(\xi) = \xi \cdot \widetilde A \xi$ with $\widetilde A$ as in \eqref{eq:lambda_for_A_nu}. In light of \eqref{eq:exp_log_eq4}, passing to the limit $K\downarrow 1$ through rationals, we finally get
\[\liminf_{t\to\infty}\frac{u^\omega(t,x)}{t^{\frac{\beta d}{2+d \beta}}}\geq
- \, \theta^{\frac{2}{2 + d\beta}} \left(\frac{1}{d}\right)^{\frac{2 \beta }{2 + d \beta}} \left(\left(\omega_d \, \rho \, \kappa^{\frac{-2}{2 + d \beta}}\right)^{2/d} \lambda_1^{(2)}(B(0,1)) + \kappa^{\frac{2 \beta }{2 + d \beta}}\right) , \quad \qpr-\text{a.s.}
\]
Again, to match the upper bound, we take $\kappa= \rho\omega_d\left(\lambda_1^{(2)}(B(0,1))\right)^{\frac{d}{2}}$. With this choice of $\kappa$  we conclude that
$$
\liminf_{t\to\infty}\frac{u^\omega(t,x)}{t^{\frac{\beta d}{2+d \beta}}}\geq
- 2 \theta^{\frac{2}{2 + d\beta}}\left(\frac{\omega_d \rho}{d}\right)^{\frac{2 \beta }{2 + d \beta}} \left(\lambda_1^{(2)}(B(0,1))\right)^{\frac{d \beta }{2 + d \beta}}, \quad \qpr-\text{a.s.},
$$
which completes the proof.
\end{proof}

We now justify the upper bound \eqref{eq:exp_log_eq2}, which is one of the key steps in the proof of the above theorem.

\begin{lemma} \label{lem:tran_dens_ub}
Let $X$ be a symmetric L\'evy process with characteristic exponent $\psi$ as in \eqref{eq:Lchexp} with $A \equiv 0$ and $\nu({\rm d}x)=\nu(x){\rm d}x$
such that there exist $\theta>0$, {$\beta >1$}, $\delta \in (0,2)$ satisfying
\begin{align} \label{eq:profile_exp_log_2}
\nu(-x) = \nu(x) \asymp \1_{\left\{|x| \leq 1 \right\}} |x|^{-d-\delta} +  \1_{\left\{|x| > 1 \right\}} e^{-\theta(\log|x|)^{\beta}}, \quad x \in \R^d \backslash \left\{0\right\}.
\end{align}
Then there exist {$C_{16} > 0$}, $r_0> 0$ and $t_1 >0$ such that
\begin{align*}
p(t,x)\leq {C_{16}} t e^{- \theta \left(\log \left(\frac{|x|}{4}\right)\right)^{\beta}}, \quad |x|>r_0 \vee t, \ t\geq t_1.
\end{align*}
\end{lemma}

\begin{proof}
We apply again \cite[Theorem 1]{bib:KS13}. First observe that the assumption (1) of this theorem holds for $f(r) = \1_{\left\{r \leq 1 \right\}} r^{-d-\delta} +  \1_{\left\{r > 1 \right\}} e^{-\theta(\log r)^{\beta}}$. We now justify its second assumption (2). As $\nu(x) \asymp f(|x|),$ we have to prove that there exists $c_1 >0$ such that
\begin{align} \label{eq:profile_est1}
\int_{|y|>r} f\big((s \vee |y|) - |y|/2 \big) f(|y|) dy \leq c_1 f(s) \Psi(1/r), \quad s,r >0,
\end{align}
where $\Psi$ is given by \eqref{eq:Lchexpprof} (see also \eqref{eq:PruitH} and comments after it). When $s \leq r$, then the integral on the left hand side can be directly estimated by $c_2 f(r) \int_{|y|>r} f(|y|/2)dy \leq c_2 f(s) \int_{|y|>r/2} f(|y|)dy  \leq c_3 f(s) \Psi(2/r) \leq c_4 f(s) \Psi(1/r)$, which is the claimed bound. When $s > r$, then we split this integral into two integrals: over $|y| \geq s$ and $r < |y| < s$, respectively. In the first case, we can follow exactly the same argument as above. Consequently, we only need to estimate the second integral $I_{s,r}:= \int_{r < |y| < s} f\big((s \vee |y|) - |y|/2 \big) f(|y|) {\rm d}y$. For $u_0={\rm e}^{\beta-1} $ we have
\begin{equation}\label{eq:uzero}
\frac{(\log u_0)^{\beta-1}}{u_0} = \sup_{u \in (1,\infty)} \frac{(\log u)^{\beta-1}}{u},
\end{equation}
and for $u>u_0$ the function $\frac{(\log u)^{\beta-1}}{u}$ is
decreasing.
We write
$$
I_{s,r} = \left(\int_{r < |y| < s \atop |y| \leq 2u_0} + \int_{r < |y| < s \atop |y| > 2u_0}\right) f\big(s - |y|/2 \big) f(|y|) dy
=: I^{(1)}_{s,r} + I^{(2)}_{s,r}
$$
(with the  convention that the integral over an empty set is equal to zero). When $s > u_0 +1$, then $I^{(1)}_{s,r} \leq c_5 f(s-u_0) \int_{|y|>r} f(|y|) dy \leq c_6 f(s) \Psi(1/r)$, {by the fact that $f(s-u_0) \leq c_7 f(s)$ for some $c_7$ uniform in $s$, and by \eqref{eq:PruitH}}. On the other hand, when $s \leq u_0 +1$, then $f$ is within doubling range and we simply have $I^{(1)}_{s,r} \leq f(s/2) \int_{|y|>r} f(|y|) dy \leq c_8 f(s) \Psi(1/r)$. To estimate $I^{(2)}_{s,r}$, we make the following observation: when $u_0<u<\frac{s}{2},$ then for $\vartheta\in(0,1)$ one has $s-\vartheta u > u \geq u_0.$
Consequently, {by the Lagrange's theorem and \eqref{eq:uzero}, for some $\vartheta \in (0,1)$, }
\[(\log s)^\beta-(\log (s-u))^{\beta} =\beta\frac{(\log(s-\vartheta u))^{\beta-1}}{s-\vartheta u}\,u \leq \beta(\log u)^{\beta-1}\] and
further  $$
(\log s)^{\beta} +\frac{1}{2} \left(\log u\right)^{\beta} \leq \left(\log\left(s-u\right)\right)^{\beta} + \left(\log (2u) \right)^{\beta},
$$
increasing $u_0$ if necessary.
This gives that $f\big(s - u) f(2u) \leq f(s) \exp(-(\theta/2) (\log(u))^{\beta})$, for the same range of $s$ and $u$.  Using these observations for $y$ in the domain of $I_{s,r}^{(2)}$ {(i.e. $u:=|y|/2$),} we get $$I^{(2)}_{s,r} \leq f(s) \int_{|y|>r \vee 2u_0} \exp(-(\theta/2) (\log(|y|/2))^{\beta}) dy.$$ Since we can directly check that the last integral is dominated by $c_9 \Psi(1/r)$, for every $r >0$, the claimed bound follows. This completes the proof of the assumption (2) of the cited theorem.

It suffices to prove the remaining condition (3). By \cite[Lemma 5 (a)]{bib:KS14}, we have $\psi(x) \asymp \Psi(|x|)$, $x \in \R^d$. Since $\Psi(r) \asymp r^{\delta} \wedge r^2$ by \eqref{eq:PruitH}, similarly as in the proof of Lemma \ref{lem:density_estimate} we can show that $\int_{\R^d} e^{-t\psi(z)} |z| dz \leq c_{11} t^{-(d+1)/2}$, for large $t$. This is exactly the missing assumption (3).

Thus, by \cite[Theorem 1]{bib:KS13} we get
$$
p(t,x) \leq c_{12} t f(|x|/4) + c_{13} t^{-d/2} e^{-c_{14} \frac{|x|}{\sqrt{t}} \log \left(1+ c_{14} \frac{|x|}{\sqrt{t}}\right)}, \quad x \in \R^d, \ t \geq t_1,
$$
for some constants $c_{12}-c_{14}$ and sufficiently large $t_1>0$. When $|x| \geq t$, the last exponential member is smaller than $c_{15} f(|x|/4)$, for some constant $c_{15} >0$. This yields the claimed upper bound for the densities.
\end{proof}

We now pass to the case when the decay of the L\'{e}vy density is {stretched exponential, exponential, or superexponential}.

\begin{theorem} \label{thm:exp}
Let $X$ be a symmetric L\'evy process with characteristic exponent $\psi$ as in \eqref{eq:Lchexp} with the Gaussian coefficient $A = (a_{ij})_{1 \leq i, j \leq d}$ such that either  $A \equiv 0$ or $\inf_{|\xi|=1} \xi \cdot A \xi > 0,$ and a symmetric L\'evy measure $\nu({\rm d}x)=\nu(x){\rm d}x$
such that there exist $\theta>0$, $\beta \in (0,\infty)$, $\gamma \geq 0$ and $\delta \in (0,2)$ such that
either
\begin{align} \label{eq:profile_exp}
\nu(x) \asymp \1_{\left\{|x| \leq 1 \right\}} |x|^{-d-\delta} + \1_{\left\{|x| > 1 \right\}} e^{-\theta (|x|-1)^{\beta}} |x|^{-\gamma}, \quad x \in \R^d \backslash \left\{0\right\},
\end{align}
or
\begin{align} \label{eq:profile_truncated}
\nu(x) \asymp \1_{\left\{|x| \leq 1 \right\}} |x|^{-d-\delta}, \quad x \in \R^d \backslash \left\{0\right\},
\end{align}
(this corresponds to the limiting case $\beta = \infty$).
Let $V^{\omega}$ be a Poissonian potential with bounded, compactly supported, nonnegative and nonidentically zero a.e. profile $W$.
Then, for any fixed $x \in \R^d$, one has
\begin{align*}
\limsup_{t \to \infty} \frac{\log u^{\omega}(t,x)}{t/(\log t)^{2/d}} \leq - \left(\frac{\rho(\beta \wedge 1)}{d}\right)^{\frac{2}{d}} \lambda_{(2)}, \quad \qpr-\text{a.s.},
\end{align*}
and
\begin{align*}
\liminf_{t\to\infty}\frac{\log u^\omega(t,x)}{t/(\log t)^{2/d}}\geq
- \left(\frac{\rho \omega_d (\beta \wedge 1)}{d}\right)^{\frac{2}{d}}\lambda_1^{(2)}(B(0,1)) , \quad \qpr-\text{a.s.},
\end{align*}
where $\lambda_1^{(2)}(U)$ and $\lambda_{(2)}$ correspond to the diffusion process with Gaussian matrix $\widetilde A$ as in \eqref{eq:lambda_for_A_nu}.

In particular, if $A = a \Id$ for some $a \geq 0$ and $\nu$ is radial nonincreasing, then
$$
\lim_{t\to\infty}\frac{\log u^\omega(t,x)}{t/(\log t)^{d/2}} =
- \left(\frac{\rho \omega_d (\beta \wedge 1)}{d}\right)^{\frac{2}{d}} \,\left(a + \frac{1}{2} \int_{\R^d} y^2_1\nu(y)dy\right) \, \lambda_1^{BM}(B(0,1)) , \quad \qpr-\text{a.s.},
$$
where $\lambda_1^{BM}(B(0,1))$ is the principal eigenvalue of the Brownian motion killed on leaving the ball $B(0,1)$.
\end{theorem}

\begin{proof}
We proceed along the same scheme as in the proof of Theorem \ref{thm:exp_log} above. By Proposition \ref{prop:suff_for_Okura} the basic assumption \textbf{(C)}  is satisfied  (with $\widetilde A$ as in \eqref{eq:lambda_for_A_nu}).

\smallskip

\textsc{The upper bound}.
We first verify the assumption \textbf{(U)}. When $A \equiv 0$, then we derive from the upper bounds in {\cite[(1.14), (1.17) and (1.21)]{bib:CKK}} that
there exist $c_1, c_2 >0$ such that
\begin{align} \label{eq:exp_eq1}
p(t,x) \leq c_1 e^{-c_2|x|^{(\beta \wedge 1)}}, \quad  \text{whenever \ $|x| \geq 2 t \geq 2$.}
\end{align}
This gives that \textbf{(U)} is satisfied with $F(r)=e^{-c_3 r^{(\beta \wedge 1)}}$, for some $c_3 \leq c_2$. By Proposition \ref{prop:gauss_est}, this also extends to the case $\inf_{|\xi|=1} \xi \cdot A \xi > 0$ (one may need to adjust constants). On the other hand, Theorem \ref{th:Okura} yields
$$
\lim_{\lambda\to 0} \lambda^{d/2}\log N^D(\lambda)=-\rho (\lambda_{(2)})^{d/2},
$$
where $\lambda_{(2)}$ is determined by the variational formula \eqref{eq:lambda} with $\lambda_1^{(2)}(U)$ corresponding to the diffusion process with exponent $\psi^{(2)}(\xi) = \xi \cdot \widetilde A \xi$, where $\widetilde A$ is given by \eqref{eq:lambda_for_A_nu}.

We are now ready to apply our general Theorem \ref{thm:upper} and Corollary \ref{cor:upper}. Recall that $\kappa_0 = \rho (\lambda_{(2)})^{d/2}$ and observe that
$$
h_{F,2,\kappa_0}(t) \approx c_4 \left(\frac{t}{(\log t)^{2/d}}\right)^{1/(\beta \wedge 1)},
$$
for some $c_4 >0$, which implies
\begin{align} \label{eq:exp_eq2}
g(t) = \frac{t}{(\log h_{F,2,\kappa_0}(t))^{2/d}} \approx \frac{(\beta \wedge 1)^{2/d} t}{(\log t)^{2/d}}.
\end{align}
Since {$Q_1 = \infty$} in  Corollary \ref{cor:upper}, we may conclude that
$$
\limsup_{t \to \infty} \frac{\log u^{\omega}(t,x)}{t/(\log t)^{2/d}} \leq - \left(\frac{\rho(\beta \wedge 1)}{d}\right)^{\frac{2}{d}} \lambda_{(2)}, \quad \qpr-\text{a.s.},
$$
which is the claimed upper bound.

\noindent\textsc{The lower bound}. Again,
by Proposition \ref{prop:conditionA}, the assumptions of our general Theorem \ref{thm:lower_bound} and Corollary \ref{cor:lower} hold
 with any $K>1$. Similarly as in the previous proofs, we consider the profile function $F(r)=e^{-c_3 r^{(\beta \wedge 1)}}$ and arbitrary $\kappa>0$. Moreover, Proposition \ref{prop:dirichlet-by-nu} gives that that there exists $c_5>0$ such that $G(1,R)\geq c_5 e^{- c_3(R/2)^{\beta}}$, for large $R$. This lower estimate is sufficiently sharp for $\beta \in (0,1]$. However, for $\beta > 1$ it is not sharp enough for our applications. Therefore, we have to address this case separately. According to the definition of the parameter function $G$ in \eqref{eq:G-def}, we derive from \cite[Propositions 3.5 and 3.6]{bib:KiKi} that there exist $c_6, c_7 >0$ such that for sufficiently large $R$
$$
G(1,R) \geq c_6 e^{-c_7 \left(\frac{R}{2}\right) \left(\log\left(\frac{R}{2}\right)\right)^{\frac{\beta-1}{\beta}}}, \quad \text{whenever \ $\beta \in (1,\infty)$,}
$$
and
$$
G(1,R) \geq c_6 e^{-c_7 \left(\frac{R}{2}\right) \log\left(\frac{R}{2}\right)}, \quad \text{in the limiting case \ $\beta = \infty$.}
$$
We are now in a position to apply Corollary \ref{cor:lower} to our main Theorem \ref{thm:lower_bound}. Observe that {$Q_1 = \infty$} in \eqref{eq:A_lower}. By using the above lower bounds for $G$, we also directly get
$$
\lim_{r \to \infty} \frac{\left|\log G\left(1,\frac{2\sqrt d r}{(\log r)^{\frac{2}{d}+2}}\right)\right|}{r\wedge|\log F(r)| + (d/2) \log r} = 0,
$$
i.e. one has {$Q_2=0$} in \eqref{eq:B_lower}. Moreover, note that the asymptotic profile $g(t)$ appearing in \eqref{eq:statement-upper:A} is $\kappa$-independent (cf. \eqref{eq:exp_eq2}). Thus, for any fixed $\kappa>0$ and $x \in \R^d$
$$
\liminf_{t \to \infty} \frac{\log u^{\omega}(t,x)}{{t/(\log t)^{2/d}}} \geq - \left(\frac{\rho \omega_d (\beta \wedge 1)}{d}\right)^{\frac{2}{d}}  \lambda_1^{(2)}(B(0,1)), \quad \qpr-\text{a.s.},
$$
which is the required lower bound.

\noindent\textsc{The concluding step}.
If $A = a \Id$ for some $a \geq 0$ and $\nu$ is radial nonincreasing, then one can show that
$$
\lambda_{(2)} = \omega_d^{2/d} \lambda_1^{(2)}(B(0,1)) = \omega_d^{2/d} \left(a + \frac{1}{2} \int_{\R^d} y^2_1\nu(y)dy\right) \lambda_1^{BM}(B(0,1)).
$$
Therefore, in this case we have
$$
\lim_{t\to\infty}\frac{u^\omega(t,x)}{t/(\log t)^{d/2}} =
- \left(\frac{\rho \omega_d (\beta \wedge 1)}{d}\right)^{\frac{2}{d}} \,\left(a + \frac{1}{2} \int_{\R^d} y^2_1\nu(y)dy\right) \,\lambda_1^{BM}(B(0,1)) , \quad \qpr-\text{a.s.}
$$
The proof is complete.
\end{proof}

\noindent

We now illustrate the above result with several important examples.

\begin{example} \rm{\textbf{(Absolutely continuous L\'evy measures with second moment finite)}. \label{ex:exp}
Our Theorem \ref{thm:exp} above immediately applies to the following examples.
\begin{itemize}
\item[(1)] \emph{Relativistic $\alpha$-stable process}. When $\psi(\xi)=(|\xi|^2+m^{2/\alpha})^{\alpha/2}-m$ with $\alpha \in (0,2)$ and $m>0$, then we have
\begin{align} \label{eq:rel1}
\psi(\xi)=\frac{\alpha}{2}m^{1-\frac{2}{\alpha}} |\xi|^2 +o(|\xi|^2) \quad \text{as} \ \ |\xi| \to 0
\end{align}
and \eqref{eq:profile_exp} holds with $\theta=m^{1/\alpha}$, $\beta=1$, $\gamma=(d+1+\alpha)/2$ and $\delta=\alpha$.
In this case, the quenched behaviour is similar to that for the Brownian motion and
we obtain precise first term asymptotics:
\begin{align}  \label{eq:rel2}
\log u^{\omega}(t,x) = \frac{\alpha}{2}m^{1-\frac{2}{\alpha}} \, \lambda_1^{BM}(B(0,1)) \left(\frac{\rho \omega_d}{d}\right)^{\frac{d}{2}} \, \frac{t}{(\log t)^{\frac{d}{2}}} + o\left(\frac{t}{(\log t)^{\frac{d}{2}}}\right), \quad \text{as \ $t \to \infty$.}
\end{align}
{It is instructive to discuss the following two limiting behaviours of the constant appearing in \eqref{eq:rel2}. When $\alpha \to 2$, then it tends to $\lambda_1^{BM}(B(0,1)) \left(\frac{\rho \omega_d}{d}\right)^{\frac{d}{2}}$, which is the constant obtained for the Brownian motion. The second limit is interesting from the mathematical physics point of view. Recall that the Hamiltonian $\sqrt{-\hbar c^2 \Delta + m^2 c^4}$ (called the \emph{Klein-Gordon square root operator} or the \emph{quasi-relativistic Hamiltonian}) is often said to describe the motion of a \emph{free quasi-relativistic particle}. Here $m$ is the mass of a particle, $c$ is the speed of light, and $\hbar$ is the reduced Planck constant. Since the term $mc^2$ represents the rest mass, the related operator $-L := \sqrt{-\hbar c^2 \Delta + m^2 c^4} - mc^2$ is often called the \emph{kinetic energy operator} (the pure jump L\'evy process generated by $L$  is called the \emph{relativistic process} and is determined by its Fourier symbol $\psi(\xi)=\sqrt{\hbar c^2 |\xi|^2 + m^2 c^4} - mc^2$). Observe that in this case \eqref{eq:rel1} reads as follows:
$$
\psi(\xi)=\frac{\hbar}{2m} |\xi|^2 +o(|\xi|^2) \quad \text{as} \ \ |\xi| \to 0.
$$
The leading term is $c$-independent and it corresponds to passing to the so-called \emph{non-relativistic limit} (i.e. $c \to \infty$). By this fact also the corresponding leading term in \eqref{eq:rel2} remains unchanged under taking such a limit (cf. \cite[Remark 1.3]{bib:KKM}).
}
\item[(2)] \emph{Isotropic tempered $\alpha$-stable process}. Let $\nu(x) = C_{d,\alpha} |x|^{-d-\alpha} e^{-m |x|^{\beta}}$ with $\alpha \in (0,2)$, $\beta>0$ and $m>0$, for some $C_{d,\alpha} >0$. In this case, one should take $\theta=m$, $\beta > 0$, $\gamma=d+\alpha$ and $\delta=\alpha$ in \eqref{eq:profile_exp}. In particular,
$$
\psi(\xi)= \left(\frac{C_{d,\alpha}}{2} \int_{\R^d} y_1^2 |y|^{-d-\alpha} e^{-m |y|^{\beta} dy } \right) |\xi|^2 +o(|\xi|^2) \quad \text{as} \ \ |\xi| \to 0.
$$
With this in mind,
\begin{align*}
\log u^{\omega}(t,x) & = \lambda_1^{BM}(B(0,1)) \left(\frac{C_{d,\alpha}}{2} \int_{\R^d} y_1^2 |y|^{-d-\alpha} e^{-m |y|^{\beta} dy } \right) \left(\frac{\rho \omega_d (\beta \wedge 1)}{d}\right)^{\frac{2}{d}} \, \frac{t}{(\log t)^{\frac{d}{2}}} \\
& \ \ \ \ \ \ \ \ \ \ \ \ \ \ \ \ \ \ \ \ \ \ \ \ \ \ \ \ \ \ \ \ \ \ \ \ \ \ \ \ \ \ \ \ \
+ o\left(\frac{t}{(\log t)^{\frac{d}{2}}}\right), \quad \text{as \ $t \to \infty$.}
\end{align*}
\item[(3)] \emph{Isotropic Lamperti stable process.} Let $\nu(x) = C_{d,\alpha} |x|^{-(d-1)} e^{m |x|} (e^{|x|} +1)^{-\alpha-1}$ with $\alpha \in (0,2)$ and $0 <m < \alpha+1$, for some $C_{d,\alpha} >0$. For this case we immediately obtain the analogous first term asymptotics as in (2).
{
\item[(4)] \emph{Truncated stable process}. Let $\nu(x) = C_{d,\alpha} |x|^{-d-\alpha} \1_{\left\{|x| \leq 1\right\}}$ with $\alpha \in (0,2)$, for some $C_{d,\alpha} >0$. This is the limiting case $\beta = \infty$. We now have
\begin{align*}
\log u^{\omega}(t,x) & = \lambda_1^{BM}(B(0,1)) \left(\frac{C_{d,\alpha}}{2} \int_{|y| \leq 1} y_1^2 |y|^{-d-\alpha} dy \right) \left(\frac{\rho \omega_d}{d}\right)^{\frac{2}{d}} \, \frac{t}{(\log t)^{\frac{d}{2}}} + o\left(\frac{t}{(\log t)^{\frac{d}{2}}}\right),
\end{align*}
as $t \to \infty$.
}
\end{itemize}
}
\end{example}

As mentioned above, for more clarity we decided to present and prove our Theorems \ref{thm:polynomial}-\ref{thm:exp} for absolutely continuous L\'evy measures only. However, we want to emphasize that similar results holds true in much more general settings.
For completeness, we now give some examples of less regular L\'evy measures to which our general Theorems \ref{thm:upper} and \ref{thm:lower_bound} apply directly. This can be justified by modification of the argument above. The details are left to the reader.

\begin{example} {\rm {\textbf{(Less regular L\'evy measures with second moment finite)}}\label{ex:ex_irr}

\noindent
\begin{itemize}
\item[(1)] {\emph{Product L\'evy measures}. Let $n$ be a symmetric finite measure on the unit sphere $S^{d-1}$ such that
$$
n(B(\varphi,r) \cap S^{d-1}) \geq c_0 r^{d-1}, \quad \varphi \in S^{d-1}, \ \ r \in (0,1/2],
$$
for some constant $c_0>0$, and let
$$
f(s):= \1_{[0,1]}(s) \cdot s^{-\theta/q}  + e^m \1_{(1,\infty)}(s) \cdot e^{-m s^{\beta}} s^{-\delta} , \quad s > 0,
$$
with $m>0$, $\beta \in (0,1/2]$, $\delta >0$ and $\theta \in (0,2q)$. Consider a symmetric L\'evy process with L\'evy-Khinchine exponent $\psi$ as in \eqref{eq:Lchexp} with diffusion matrix $A$ such that $A \equiv 0$ or $\inf_{|\xi|=1} \xi \cdot A \xi > 0$ and product L\'evy measure $\nu({\rm d}r {\rm d}\varphi) = n({\rm d} \varphi) f(r) {\rm d}r$. Then the $\qpr$-a.s. bounds for $\liminf_{t \to \infty} \frac{\log u^{\omega}(t,x)}{t/(\log t)^{2/d}}$ and $\limsup_{t \to \infty} \frac{\log u^{\omega}(t,x)}{t/(\log t)^{2/d}}$ of Theorem \ref{thm:exp} extend to this case. Note that we do not impose any growth condition on $n$ from above. Therefore this example covers a wide range of L\'evy measures that are not absolutely continuous with respect to the Lebesgue measure. Many other examples can also be produced by changing the profile $f$.}

\item[(2)] \emph{L\'evy measures with purely discrete long jumps parts}. Let $f:(-1,1)^d \cup \Z^d \to \R$ be given by
$$
f(x) = \left\{\begin{array}{ll}
\frac{1}{2^{(d+\theta) n}}  &\mbox{when } x \in \left(-\frac{1}{2^n}, \frac{1}{2^n} \right)^d \setminus \left[-\frac{1}{2^{n+1}}, \frac{1}{2^{n+1}}\right]^d, \ n \in \Z , \\
\frac{1}{(\max_{1 \leq i \leq d} |x_i|)^{d+\delta}} & \mbox{when } x = (x_1,x_2,...,x_d) \in \Z^d \setminus \left\{0\right\},
\end{array}\right.
$$
with $\theta \in (0,2)$ and $\delta >d$. Denote by $f_r(x) = f(rx)$, $r>0$, the dilatations of $f$. Consider a symmetric L\'evy process with L\'evy-Khinchine exponent $\psi$ as in \eqref{eq:Lchexp} with diffusion matrix $A$ such that $A \equiv 0$ or $\inf_{|\xi|=1} \xi \cdot A \xi > 0$ and L\'evy measure $\nu_r$ defined by
$$
\nu_r(B):= \int_{B \cap (-1,1)^d} f_r(y) dy + \sum_{y \in B \cap \Z^d \setminus \left\{0\right\}} f_r(y),
$$
for every Borel set $B \subset \R^d$ and for given $r>0$. Then the $\qpr$-a.s. bounds for $\liminf_{t \to \infty} \frac{\log u^{\omega}(t,x)}{t^{d/(d+2)}}$ and $\limsup_{t \to \infty} \frac{\log u^{\omega}(t,x)}{t^{d/(d+2)}}$ with $\delta_1 = \delta$ and $\delta_2 = \delta-d$ as in Theorem \ref{thm:polynomial} also apply.
\end{itemize}
}
\end{example}


\begin{thebibliography}{00}

\bibitem{bib:App}
D. Applebaum:
\emph{L\'evy Processes and Stochastic Calculus},
Cambridge University Press, 2nd ed., 2008



\bibitem{BGR13} K. Bogdan, T. Grzywny, M. Ryznar, {\it Density and tails of unimodal convolution semigroups}, J. Funct. Anal. 266 (6) (2014), 3543-3571.

\bibitem{BBKRSV} K. Bogdan et al, {\it Potential Analysis of Stable Processes and its Extensions} (ed. P. Graczyk, A. St\'os),
Lecture Notes in Mathematics 1980, Springer, Berlin, 2009.

\bibitem{bib:blp} L. Brasco, E. Lindgren, E. Parini, {\em  The fractional Cheeger problem.} Interfaces
Free Bound. 16 (2014), 419-458. doi: 10.4171/IFB/325

\bibitem{bib:CL}  R. Carmona, J. Lacroix, {\em Spectral theory of random Schr\"{o}dinger operators}. Probability and its Applications. Birkh\"auser Boston, Inc., Boston, MA, 1990.

\bibitem{bib:CMS90}
R. Carmona, W.C. Masters, B. Simon,
\emph{Relativistic Schr\"odinger operators: asymptotic behaviour of the eigenfunctions},
J. Funct. Anal. 91, 1990, 117-142.

\bibitem{bib:CKK}
Z.-Q. Chen, P. Kim, T. Kumagai,
\emph{Global heat kernel estimates for symmetric jump processes},  Trans. Amer. Math. Soc.
363 (9), 2011, 5021--5055.


\bibitem{bib:CS0}
Z.-Q. Chen, R. Song,
\emph{Two sided eigenvalue estimates for subordinate processes in domains}, J. Funct. Anal. 226, 2005, 90-113.

\bibitem{DC}
M. Demuth, J. A. van Casteren,
\emph{Stochastic Spectral Theory for Self-adjoint Feller Operators. A Functional Analysis Approach}, Birkh\"auser, Basel, 2000.




\bibitem{bib:Don-Var} M.D. Donsker, S.R.S. Varadhan, {\em Asymptotics for the Wiener sausage},
 Comm. Pure Appl. Math. 28 (1975), no. 4, 525--565.

\bibitem{bib:Fuk} R. Fukushima,
{\emph From the Lifshitz tail to the quenched survival asymptotics in the trapping problem}, Electron. Commun. Probab. 14 (2009), 435--446.

\bibitem{bib:Grz}
T. Grzywny,
\emph{On Harnack Inequality and H\"older Regularity for Isotropic Unimodal L\'evy Processes},
Potential Anal. 41, 1--29 (2014).

\bibitem{bib:GrzSz}
T. Grzywny, K. Szczypkowski,
\emph{Kato classes for L\'evy processes}, preprint, 2015, available at arXiv:1503.05747.

\bibitem{bib:IW}
N. Ikeda, S. Watanabe,
\emph{On some relations between the harmonic measure and the L\'evy measure for a certain class of Markov processes},
J. Math. Kyoto Univ. 2-1, 1961, 79-95,

\bibitem{bib:J}
N. Jacob,
\emph{Pseudo-Differential Operators and Markov Processes: Markov Processes and Applications},
vols. 1-3, Imperial College Press, 2003-2005.

\bibitem{bib:JSch} N. Jacob, R. L. Schilling, {\it L\'{e}vy-type processes and pseudo differential operators}. In: O. Barndorff-Nielsen, T. Mikosch and S. Resnick (eds.): L\'{e}vy processes: theory and applications , Birkh\"{a}user, Boston 2001, 139 -167.




\bibitem{bib:KKM}
K. Kaleta, M. Kwa\'snicki, J. Ma\l ecki,
\emph{One-dimensional quasi-relativistic particle in the box}, Rev. Math. Phys. 25, 1350014, 2013.

\bibitem{bib:KL14}
K. Kaleta, J. L\H{o}rinczi,
\emph{Pointwise eigenfunction estimates and intrinsic ultracontractivity-type properties of Feynman-Kac semigroups for a class of L\'evy processes}, Ann. Probab. 43 (3), 2015, 1350-1398.

\bibitem{bib:KL15}
K. Kaleta, J. L\H{o}rinczi,
\emph{Fall-off of eigenfunctions for non-local Schr\"odinger operators with decaying potentials}, preprint 2015, available at arXiv:1503.03508.

\bibitem{bib:KS13}
K. Kaleta, P. Sztonyk,
\emph{Estimates of transition densities and their derivatives for jump Levy processes}, J. Math. Anal. Appl. 431 (1), 2015, 260-282.

\bibitem{bib:KS14}
K. Kaleta, P. Sztonyk,
\emph{Small time sharp bounds for kernels of convolution semigroups}, J. Anal. Math., to appear, available at arXiv:1403.0912.






\bibitem{bib:Kal-Pie-SPA}
K. Kaleta, K. Pietruska-Pa\l uba, \emph {Integrated density of states for Poisson-Schr\"{o}dinger processes on the Sierpi\'{n}ski gasket},
Stochastic Process. Appl. 125 (4), 2015, 1244-1281.


\bibitem{bib:kk-kpp-lif} K. Kaleta, K. Pietruska-Pa\l uba, {\em Lifschitz singularity for subordinate Brownian motions in presence of the Poissonian potential on the Sierpi\'{n}ski gasket}, preprint 2014,
available at arXiv:1410.1715.

\bibitem{bib:KiKi}
K.-Y. Kim, P Kim,
\emph{Two-sided estimates for the transition densities of
symmetric Markov processes dominated by stable-like
processes in $C^{1,\eta}$ open sets}, Stochastic Process. Appl. 124 (9), 2014, 3055–3083.

\bibitem{bib:Kon-Wolff} W. K\"{o}nig, {\em The parabolic Anderson model}, preprint (2015), available at http://www.wias-berlin.de/people/koenig/www/PAMsurveyBook.pdf

\bibitem{bib:Okura} H. Okura, {\em On the spectral distributions of certain integro-differential operators with random potential}, Osaka J. Math. 16 (1979), no. 3, 633-666.

\bibitem{bib:Okura81} H. Okura, {\em Some limit theorems of Donsker-Varadhan type for Markov processes expectations}, Z. Wahrscheinlichkeitstheorie verw. Gebiete 57 (1981), 419-440.

\bibitem{bib:kpp-ptrf}
K. Pietruska-Pa{\l}uba, {\em The Lifschitz singularity for the density of states on the Sierpiński gasket},  Probab. Theory Related Fields 89 (1991), no. 1, 1-33.

\bibitem{bib:kpp-spa}
K. Pietruska-Pa{\l}uba,
{\em Almost sure behaviour of the perturbed Brownian motion on the Sierpiñski gasket.} Stochastic Process. Appl. 85 (2000), no. 1, 1-17.

\bibitem{bib:Pru}
W.E. Pruitt,
\emph{The growth of random walks and L\'evy processes},
Ann. Probab. 9 no. 6, 948--956 (1981).

\bibitem{bib:Sat}
K.-I. Sato,
\emph{L{\'e}vy Processes and Infinitely Divisible Distributions},
Cambridge Univ. Press, Cambridge, 1999.

\bibitem{bib:Szn-ptrf93} A.S. Sznitman, {\em Brownian asymptotics in a Poissonian environment.}
Probab. Theory Relat. Fields 95 (1993), 155-174.

\bibitem{bib:SSV}
R. Schilling, R. Song, Z. Vondra\v{c}ek,
\emph{Bernstein Functions}, Walter de Gruyter, 2010.

\bibitem{bib:Sch}
R. Schilling,
\emph{Growth and H\"older conditions for the sample paths
of Feller processes}, Probab. Theory Relat. Fields 112, 565--611 (1998).

\bibitem{bib:Schm}
K. Schm\"udgen,
\emph{Unbounded Self-adjoint Operators on Hilbert Space},
Graduate Texts in Mathematics 265, Springer 2012.






\end{thebibliography}
\end{document}